\documentclass[11pt,twoside,reqno,psamsfonts]{amsart}

\usepackage[left=2.9cm,top=3cm,right=2.9cm]{geometry}              
\usepackage[colorlinks = true, citecolor = black, linkcolor = black, urlcolor = black, pdfstartview=FitH]{hyperref}  
\usepackage[pdftex]{graphicx}
\usepackage[utf8]{inputenc}
\usepackage{microtype, float}
\usepackage{amsmath, verbatim}
\usepackage{amssymb}
\usepackage{epstopdf}
\usepackage{epsfig}
\usepackage{mathscinet}

\let\oldtocsection=\tocsection
\let\oldtocsubsection=\tocsubsection
\let\oldtocsubsubsection=\tocsubsubsection
\renewcommand{\tocsection}[2]{\hspace{0em}\textbf{\oldtocsection{#1}{#2}}}
\renewcommand{\tocsubsection}[2]{\hspace{1.8em}\oldtocsubsection{#1}{#2}}
\renewcommand{\tocsubsubsection}[2]{\hspace{3em}\oldtocsubsubsection{#1}{#2}}

\DeclareUnicodeCharacter{00A0}{ } 

\numberwithin{equation}{section}

\theoremstyle{plain}
\newtheorem{thm}{Theorem}[section]
\newtheorem{theorem}[thm]{Theorem}
\newtheorem{conj}[thm]{Conjecture}
\newtheorem{lemma}[thm]{Lemma}
\newtheorem{prop}[thm]{Proposition}

\theoremstyle{definition}

\theoremstyle{remark}
\newtheorem{remark}[thm]{Remark}

\newcommand{\R}{\mathbb{R}}

\newcommand{\C}{\mathbb{C}}

\newcommand{\N}{\mathbb{N}}

\renewcommand{\P}{\mathbb{P}}

\renewcommand{\H}{\mathbb{H}}
\renewcommand{\S}{\mathbb{S}}
\newcommand{\cB}{\mathcal{B}}

\newcommand{\cM}{\mathcal{M}}
\newcommand{\cH}{\mathcal{H}}

\newcommand{\cE}{\mathcal{E}}

\newcommand{\cR}{\mathcal{R}}
\newcommand{\cC}{\mathcal{C}}

\newcommand{\cT}{\mathcal{T}}

\newcommand{\cS}{\mathcal{S}}
\newcommand{\cA}{\mathcal{A}}

\newcommand{\eps}{\varepsilon}

\newcommand{\Z}{\mathbb{Z}}
\renewcommand{\emptyset}{\varnothing}
\renewcommand{\epsilon}{\varepsilon}
\renewcommand{\rho}{\varrho}
\renewcommand{\phi}{\varphi}

\newcommand{\sys}{\mathrm{sys}}

\newcommand{\1}{\mathrm{\mathbf{1}}}

\renewcommand{\hat}{\widehat}

\renewcommand{\iint}{\int\hspace{-0.1in}\int}

\newcommand{\HS}{\mathrm{HS}}
\renewcommand{\Im}{\mathrm{Im\,}}
\renewcommand{\Re}{\mathrm{Re\,}}

\newcommand{\id}{\operatorname{id}}
\renewcommand{\mod}{\,\,\mathrm{mod}\,}

\DeclareMathOperator{\spt}{spt}
\DeclareMathOperator{\Tr}{Tr}

\DeclareMathOperator{\Vol}{Vol}
\DeclareMathOperator{\Var}{Dev}

\DeclareMathOperator{\inj}{inj}
\DeclareMathOperator{\SL}{SL}
\DeclareMathOperator{\PSL}{PSL}

   \def\XXint#1#2#3{{\setbox0=\hbox{$#1{#2#3}{\int}$}
        \vcenter{\hbox{$#2#3$}}\kern-.5\wd0}}

\usepackage[usenames,dvipsnames]{xcolor}

\title[Quantum ergodicity for Eisenstein series]{Quantum ergodicity for Eisenstein series on hyperbolic surfaces of large genus}

\author{Etienne Le~Masson}
\address{Laboratoire de Mathématiques AGM, UMR CNRS 8088, CY Cergy Paris Université, 2 Av. Adolphe Chauvin, 95302 Cergy-Pointoise Cedex, France}
\email{etienne.le-masson@cyu.fr}

\author{Tuomas Sahlsten}
\address{Department of Mathematics and Systems Analysis, Aalto University, Finland \& Department of Mathematics, University of Manchester, Oxford Rd, Manchester M13 9PL, United Kingdom}
\email{tuomas.sahlsten@aalto.fi, tuomas.sahlsten@manchester.ac.uk}


\thanks{E. Le Masson was partially supported by the Marie Sk{\l}odowska-Curie Individual Fellowship grant $\sharp$703162,  by a Rutherford fellowship at the University of Warwick, and by Initiative d'Excellence Paris//Seine. T. Sahlsten was partially supported by the Marie Sk{\l}odowska-Curie Individual Fellowship grant $\sharp$655310, a start-up fund in the School of Mathematics of the University of Manchester and Academy of Finland Research Fellowship grants $\sharp$347365 and $\sharp$353738.}

\begin{document}

\begin{abstract}
We give a quantitative estimate for the quantum mean absolute deviation on hyperbolic surfaces of finite area in terms of geometric parameters such as the genus, number of cusps and injectivity radius. It implies a delocalisation result of quantum ergodicity type for eigenfunctions of the Laplacian on hyperbolic surfaces of finite area that Benjamini-Schramm converge to the hyperbolic plane. We show that this is generic for Mirzakhani's model of random surfaces chosen uniformly with respect to the Weil-Petersson volume. Depending on the particular sequence of surfaces considered this gives a result of delocalisation of most cusp forms or of Eisenstein series.
\end{abstract}

\maketitle

\section{Introduction}

\subsection{Delocalisation of eigenfunctions} 

The question of the delocalisation of eigenfunctions is a widely studied topic in hyperbolic geometry. One of the main results on this topic is the Quantum Ergodicity theorem. Let $X$ be a compact hyperbolic surface (or more generally a compact manifold with ergodic geodesic flow). Denote by $\Delta$ the Laplacian acting on $L^2(X)$  and by $\lambda_j$ the non-decreasing sequence of eigenvalues of $\Delta$. The Quantum Ergodicity theorem of Snirelman, Zelditch and Colin de Verdière \cite{Sni74, Zel87, CdV85} asserts that for any orthonormal basis of eigenfunctions $\psi_j$ in $L^2(X)$, we can find a subsequence of density $1$ of the probability measures $|\psi_j(z)|^2 \, d\Vol(z)$ that weakly converge to the normalised Riemannian volume measure $\frac1{\Vol(X)} d\Vol(z) $ when $\lambda_j \to +\infty$. Quantum Ergodicity can be alternatively formulated by studying the \textit{quantum variance} for any continuous $a: X \to \R$
$$ \frac1{\# \{j : \lambda_j \leq \lambda\}} \sum_{j : \lambda_j \leq \lambda} \left| \int_X a(z) \left( |\psi_j(z)|^2 - \frac1{\Vol(X)} \right) d\Vol(z) \right|^2 $$
and showing that it tends to $0$ when $\lambda \to +\infty$. The idea is that by this convergence we obtain an equidistribution of eigenfunctions on average over the spectrum. On hyperbolic manifolds, this quantum ergodicity property has also been shown to hold by the authors when averaging on a bounded spectral interval, and making the volume of $X$ tend to infinity instead \cite{LS17} (see also \cite{ABLM} for dimension $> 2$). We will call this type of setting the \emph{level aspect}, as opposed to the \emph{eigenvalue aspect}, i.e. the limit $\lambda_j \to +\infty$.

In the eigenvalue aspect, Zelditch proved that the Quantum Ergodicity property extends to non-compact hyperbolic surfaces of finite area \cite{Zel91}, which will be the focus of this article. Since the surface $X$ is non-compact there is both discrete and continuous spectra for the Laplacian. Let $\lambda_0 = 0 < \lambda_1 \leq \lambda_2 \leq \dots$ be the discrete spectrum and fix a corresponding orthonormal system $\{\psi_j\}_{j \in \N}  \subset L^2(X)$ of eigenfunctions of the Laplacian. The continuous spectrum is the interval $[\frac14, +\infty)$.  We denote by $\mathfrak{C}(X)$ the set of cusps on $X$. Given $r \in \R$ and $\mathfrak{b} \in \mathfrak{C}(X)$, there is (non-$L^2$) eigenfunction of the Laplacian $E_\mathfrak{b}(\cdot,\tfrac{1}{2}+ ir) : X \to \C$ called Eisenstein series,  with eigenvalue $\tau(r) = \frac{1}{4}+ r^2,$ see for example \cite{Iwa02} for background. We similarly parametrise the discrete eigenvalues $\tau(r_j) = \lambda_j$, with $r_j$ possibly complex.

Let now $I \subset (\frac{1}{4},+\infty)$ be an arbitrary interval. 
We let $N(X,I)$ be the number of (discrete) eigenvalues $\lambda_j$ of the Laplacian on $X$ which are in $I$ including multiplicities, and
$$M(X,I) := \frac{1}{4\pi} \int_{\tau^{-1}(I)}\frac{-\phi_X'}{\phi_X}\left(\tfrac{1}{2}+ ir \right)\, dr$$
where $\phi_X(s)$ is the determinant of the scattering matrix, see Section \ref{sec:prelim} or \cite{Iwa02} for details. Then the sum $N(X,I) + M(X,I)$ measures the total contribution of the discrete and continuous spectra in the interval $I$. This definition was also used for intervals $I_T = [\frac{1}{4},\frac{1}{4}+T^2]$, $T > 0$, by Zelditch in \cite{Zel91} in the study of Quantum Ergodicity of Eisenstein series in the semiclassical limit $T \to +\infty$.

 We define the  \emph{quantum mean absolute deviation} over $I$ of the eigenfunctions by
\begin{multline}\label{e:qvariance}
\Var_{X,I}(a) = \frac1{N(X,I) + M(X,I)} \left( \sum_{\lambda_j \in I} \left| \langle \psi_j, a \, \psi_j \rangle - \bar a \right| \right. \\ \left. + \frac1{4\pi} \int\limits_{\tau^{-1}(I)} \left|  \sum_{\mathfrak{b}\in \mathfrak{C}(X)} \langle E_{\mathfrak{b}}(\cdot,\tfrac12 + ir), a \, E_{\mathfrak{b}}(\cdot,\tfrac12 + ir) \rangle + \frac{\varphi_X'}{\varphi_X} \left( \frac12 + ir \right) \bar a \,\right| \, dr \right),
\end{multline}
where $a \in L^\infty(X)$ is compactly supported,
$$\langle \psi , a \, \psi \rangle = \int_X a(z)  |\psi(z)|^2 \, d\mu(z) \quad \text{and} \quad \bar{a} =  \frac1{\Vol(X)} \int_X a(z)\, d\mu(z). $$
The quantity $\Var_{X,I} (a)$ measures how far the $L^2$-mass (localised by $a$) of typical eigenfunctions and Eisenstein series is from being uniformly distributed.

Zelditch proved in \cite{Zel91} that when $I_T = [\frac{1}{4},\frac{1}{4}+T^2]$, then for any smooth compactly supported test function $a: X \to \C$, $\Var_{X,I_T}(a) \to 0$ when $T \to +\infty$. In this paper, we are interested in estimating $\Var_{X,I}$ for a fixed bounded interval $I$, in terms of geometric parameters of $X$. We prove in particular that under natural assumptions, $\Var_{X,I}(a) \to 0$ when $\Vol(X) \to +\infty$, providing the level aspect counterpart of Zelditch's result. Note that because of the presence of the Einsenstein series we only know how to deal with the quantum mean absolute deviation instead of the quantum variance (where the absolute value in the sum and the integral would be squared). This is similar to the situation in the eigenvalue aspect \cite{Zel91}, where also quantum mean absolute deviation is used.

Before we state our estimate let us introduce some definitions. We see a hyperbolic surface $X = \Gamma \backslash \H$ as a quotient of the hyperbolic plane by a discrete group $\Gamma$ of isometries. We denote by $\inj_X$ the radius of injectivity of $X$, that is, and 
$\inj_{X} = \inf_{z \in X} \inj_X(z)$, where
$$ \inj_X(z) = \frac12 \inf\{d(z,\gamma z) \: | \: \gamma \in \Gamma - \{\id\} \}, $$
and by $(X)_{\leq R}$ the $R$-thin part, i.e. the set
$$(X)_{\leq R} = \{ z\in X : \inj_X(z) \leq R \}.$$
Given $Y >0$ we can divide the surface $X$ into a compact part where the cusps are cut at a height $Y$, and a non-compact cuspidal part: the compact part $X(Y)$ is the complement of the cuspidal part and is defined by
$$ X(Y) := X \setminus \bigcup_{\mathfrak{b}} X_{\mathfrak{b}} (Y),$$
where $X_{\mathfrak{b}} (Y)$ is the cuspidal zone associated with the cusp $\mathfrak{b}$ (See the background in Section \ref{s:cusps} for details or \cite[Section 2.2]{Iwa02}). All cuspidal zones are isometric. We also use the notation $L^\infty_Y(X)$ for test functions $a \in L^\infty(X)$ such that the support of $a$ satisfy $\spt a \subset X(Y)$. Throughout the paper, we will write $A \lesssim_D B$ to denote that $A \leq C_D B$ with a constant $C_D$ depending on $D$.

Using all these parameters, we can now state our main quantitative geometric estimate for the quantum mean absolute deviation:

\begin{thm}\label{thm:main} Fix $I \subset (\frac{1}{4},+\infty)$ a compact interval. Then there exists $R_I > 0$ such that for all $R > R_I$, $k \in \N$ and $Y > 0$ the following holds. Assume $X$ is a finite area hyperbolic surface with $k$ cusps. For any $a \in L^\infty_Y(X)$, we have
\begin{align*}
 \widetilde \Var_{X,I}(a) 
& \lesssim_I    \Big[\max\{N(X,I), k\}^{1/2}\Big( \frac{ \Vol(X)}{ \,\rho(\lambda_1(X)) R} +\frac{ e^{2R}}{\min\{1,{\inj_{X(Y e^{R/2})}}^2 \}} \Vol((X)_{\leq R} ) \Big)^{1/2}  \\
&\quad + \Big( 2k \log Y + k^2 e^{-4\pi Y} +  \frac{k}{\Vol(X)} \Big(M(X,I) + k \log\Vol(X) \Big) \Big)\Big]\|a\|_\infty,
\end{align*}
where $\widetilde \Var_{X,I}(a)  = (N(X,I) + M(X,I))\Var_{X,I}(a)$ and $\rho(\lambda_1(X))$ is a function of the spectral gap of $X$
\end{thm}

Now, this quantitative estimate can be used to give a Quantum Ergodicity type theorem for a sequence of surfaces $(X_n)$. We say that a sequence of finite area hyperbolic surfaces $X_n$ \textit{Benjamini-Schramm converges} to $\H$ if for any $R > 0$,
$$ \frac{\Vol((X_n)_{\leq R} )}{\Vol(X_n)} \to 0$$
when $n \to +\infty$. Under Benjamini-Schramm converging $X_n \to \H$, we obtain the following Quantum Ergodicity theorem for eigenfunctions:

\begin{thm}\label{thm:mainBS} Let $I \subset (\frac{1}{4},+\infty)$ be a compact interval. Let $X_n = \Gamma_n\backslash \H$ be a sequence of finite area hyperbolic surface that Benjamini-Schramm converge to $\H$. Assume in addition that
\begin{enumerate}
\item $X_n$ has a uniform spectral gap (the first non-zero eigenvalue of the Laplacian is bounded away from $0$ uniformly in $n$);
\item The systole (length of the shortest closed geodesic) of $X_n$ is bounded uniformly from below;
\item The number of cusps $k_n = k(X_n)$ of $X_n$ satisfies for some $0 \leq \kappa < 1/2$, $k_n = O(g_n^\kappa)$ when $n\to +\infty$, where $g_n = g(X_n)$ is the genus of $X_n$.
\end{enumerate} 
Fix $Y > 0$ and let $(a_n)_{n\in\N}$ be a uniformly bounded sequence of measurable functions such that $\spt a_n \subset X_n(Y)$. We have
\begin{align}\Var_{X_n,I}(a_n) \to 0\label{e:qeBS}\end{align}
when $n \to +\infty$.
\end{thm}

Note that the proof gives a stronger statement than Theorem \ref{thm:mainBS}, where we are able to let the systole and the spectral gap shrink to $0$ and the support of the test functions expand, provided that all of this happens slowly enough when $n \to +\infty$. 

\begin{remark}
In the above results, we assume $I \subset (1/4,\infty)$ because Lemma \ref{lma:spectralaction} used in the proof of Theorem \ref{thm:main} needs $a > \frac{1}{4}$, but we believe the result should hold also for intervals intersecting $[0,1/4)$ with more careful analysis of the spectral action of the propagator.
 \end{remark}

The proof of Theorem \ref{thm:mainBS} follows from Theorem \ref{thm:main} provided that we can control the asymptotic behaviour of $N(X_n,I) + M(X_n,I)$ as $X_n$ Benjamini-Schramm converges to $\H$, which is given by the following spectral convergence theorem that we prove in Section \ref{s:spectralconv}.

\begin{thm}\label{thm:weyl}
 Let $I \subset [0,+\infty)$ be a compact interval and $X_n$ a sequence of hyperbolic surfaces of finite area that Benjamini-Schramm converges to the hyperbolic plane $\H$ and such that the systole is uniformly bounded from below. Then
\begin{equation}\label{eq:bsweyl}
N(X_n,I) + M(X_n,I) \sim \Vol(X_n)
\end{equation}
when $n \to +\infty$.
\end{thm}

\begin{remark} As shown in \cite{ABLM}, the methods of \cite{LS17} can be extended to higher dimensional compact hyperbolic manifolds. This is because Selberg's theory (spectral side of the proof) and the quantitative ergodic theorem of Nevo (geometric side of the proof) extend naturally to more general symmetric spaces. We expect that the new elements we introduce in this paper can also be generalised to finite volume hyperbolic manifolds of any dimension. Concerning variable curvature cusp manifolds --- for a level aspect analogue of \cite{BZ16} --- the main difficulties lie already in proving a version of the theorem for compact variable curvature manifolds. In this case indeed, we cannot use Selberg's theory and Nevo's ergodic theorem. We would need to use lower level tools such as estimates on wave propagation and exponential mixing of the geodesic flow.
\end{remark}

Let us now discuss about some consequences of Theorems \ref{thm:main}, \ref{thm:mainBS} and \ref{thm:weyl}.

\subsection{Equidistribution of Maass forms in the level aspect}

On non-compact finite area surfaces the existence of a discrete sequence of eigenvalues in $L^2$ is not guaranteed and is in fact believed to rarely happen (See \cite{PS85a,PS85b}). Our general result therefore can mostly be seen as an equidistribution result for Eisenstein series. However, in the case of the modular surface, $\Gamma = \SL(2,\Z)$, a discrete spectrum is known to exist since the work of Selberg. More generally, this is the case for any congruence subgroup defined by
\begin{equation}\label{e:congruence}
 \Gamma_0(N) = \left\{ \begin{pmatrix} a & b \\ c & d  \end{pmatrix} \in \SL(2,\Z) : c \equiv 0 \mod  N \right\}.
 \end{equation}
In this setting, relevant in number theory, eigenfunctions are usually called \textit{Maass forms}. The arithmetic structure carries a family of operators called Hecke operators that commute with the Laplacian and it was proved by Lindenstrauss \cite{Lin06} and Soundararajan \cite{Sou10} that joint eigenfunctions of these operators and the Laplacian satisfy quantum unique ergodicity. This property implies the equidistribution of all eigenfunctions in the large eigenvalue limit (see for example \cite{Ber16} for an introduction to these questions). 

The level aspect limit in the arithmetic setting was considered recently in a series of paper concerning the equidistribution of holomorphic forms by Nelson \cite{Nel11} and Nelson, Pitale and Saha \cite{NPS14}. The results are analogous to the quantum unique ergodicity theory but they rely on the proof of the Ramanujan conjectures which is not available for Maass forms. 
  
It turns out that the surfaces $Y_0(N) = \Gamma_0(N)\backslash \H$ Benjamini-Schramm converge to $\H$ when the level $N \to \infty$ (\cite{7samurai, Ra}). Moreover, in the case of increasing congruence covers, Finis, Lapid and M\"uller showed that the discrete spectrum dominates the asymptotics as the level $N \to \infty$ (see \cite{FLM15}). This means that $\frac{M(Y_0(N),I)}{N(Y_0(N),I)} \to 0$ when $N \to +\infty$. Hence Theorem \ref{thm:main} together with Theorem \ref{thm:weyl} implies a Quantum Ergodicity theorem for Maass cusp forms, which incidentally does not need to assume the cusp forms are Hecke eigenfunctions. For Hecke-Maass cusp forms, however, a quantum ergodicity theorem with a stronger rate of convergence has recently been obtained by Nelson \cite{Nel19}.

\subsection{Quantum ergodicity on random surfaces of large genus}\label{sec:introrandom}

The quantitative estimate in Theorem \ref{thm:main} allows us to study the delocalisation of eigenfunctions on random surfaces of large genus, where the random model we will use is the probability density with respect to the Weil-Petersson volume on the moduli space of hyperbolic surfaces of finite volume. This probability model for hyperbolic surfaces was popularised by the work of Mirzakhani \cite{Mi} (see also the survey by Wright \cite{Wright}) and provides very effective tools to estimate the geometric parameters appearing in Theorem \ref{thm:main}.

The approach for delocalisation of eigenfunctions was introduced in \cite{GLMST19} in the study of $L^p$ norms of eigenfunctions on random surfaces in the Weil-Petersson model (see also the recent work of Monk \cite{Monk}). Furthermore, in our case, we will also rely on some of the recent work on the spectral gaps $\lambda_1(X)$ of random surfaces $X$ of large genus, in particular the work of Hide \cite{Hide21} on spectral gaps of random finite area surfaces. Hide's work continues and builds the highly active area on bounding the spectral gap $\lambda_1(X)$ for various models of random surfaces $X$ with probability going to $1$ as the genus of $X$ grows \cite{MNP20,WuXue,LipnowskiWright}. We also highlight further related works such as Magee-Naud \cite{MN20} in the non-compact case and a recent related breakthrough of Hide and Magee \cite{MageeHide} on optimal spectral gap on a sequence of finite area surfaces.

To fix some notation, denote by $\Sigma_{g,k}$ a topological surface of genus $g \in \N$ with $k \geq 0$ punctures, which we associate with cusps. In the case $k = 0$, the surface has no punctures and thus is compact. Let then $\cT(\Sigma_{g,k})$ be the corresponding Teichm\"uller space of Riemann surface structures on $\Sigma_{g,k}$ identified up to an isotopy. Then the \textit{moduli space} $\cM_{g,k}$ is the quotient $\cM_{g,k} = \cT(\Sigma_{g,k}) / \mathrm{MCG}(\Sigma_{g,k})$, where $\mathrm{MCG}(\Sigma_{g,k})$ is the mapping class group of isotopy classes of homeomorphisms of $\Sigma_{g,k}$. Then $\cM_{g,k}$ is independent of the base surface $\Sigma_{g,k}$ chosen. There is a canonical symplectic form $\omega_{g,k}$ invariant under the mapping class group on $\cT(\Sigma_{g,k})$. This form  $\omega_{g,k}$ gives rise to the \textit{Weil-Petersson volume} $\Vol_{g,k}$ on $\cM_{g,k}$, see \cite{Wolpert83}. The Weil-Petersson volume $\Vol_{g,k}(\cM_{g,k})$ of $\cM_{g,k}$ is a constant multiple of $\pi^{6g - 6 + 2k}$, see e.g. Wolpert's work \cite{Wolpert83}, so we can normalise $\Vol_{g,k}$ to define a probability measure $\P_{g,k}$ on $\cM_{g,k}$. This is our probability model. We refer to Mirzakhani \cite{Mi} and Wright \cite{Wright} for more background and notation.

\begin{thm}\label{thm:largegenus} Let $k(g) \in \N$ be such that $k(g) = O(g^{\kappa})$ for some $0 \leq \kappa < 1/2$ as $g \to \infty$, then for $g$ large enough and a $\P_{g,k(g)}$-random surface $X \in \cM_{g,k(g)}$, we have that for any $a \in L^\infty_{\log g}(X)$ and any compact interval $I \subset (\frac{1}{4},+\infty)$ that
$$\Var_{X,I}(a) \lesssim_{I,\kappa} \frac{1}{\sqrt{\log g}} \|a\|_\infty$$
with probability at least $1- O(g^{-\beta})$ for some power $\beta > 0$ depending on $\kappa$.
\end{thm}

In other words, provided we control the number of cusps and the support of the test function, the quantum mean absolute deviation tends to $0$ with high probability when $g \to +\infty$ at a rate of $O((\log g)^{-1/2})$. The logarithmic rate that we obtain is analogous to the rates obtained by Zelditch \cite{Zel94} in the large eigenvalue limit on compact hyperbolic surfaces, and Schubert \cite{Sch06} in the semiclassical setting.

The proof of Theorem \ref{thm:largegenus} follows from Theorem \ref{thm:main} together with the following properties of random surfaces.  Fix $\eps > 0$, $g \geq 2$ and $k(g) \in \N$ such that $k(g) = O(g^{\kappa})$ for some $0 \leq \kappa < 1/2$. Let $\cB_{\eps,\kappa,g,k(g)} \subset \cM_{g,k(g)}$ be the set of surfaces such that
\begin{itemize}
\item[(i)] the thin part satisfies $$\frac{\Vol\left((X)_{\leq \frac16 \log g}\right)}{\Vol(X)} \leq g^{-\frac13},$$
\item[(ii)]the systole satisfies $$\mathrm{sys}(X) \geq g^{-\frac1{24}}(\log g)^{\frac12},$$
\item[(iii)] the spectral gap satisfies
$$\lambda_1(X) \geq \frac{1}{4} - \Big(\frac{2\kappa+1}{4}\Big)^2 - \eps.$$
\end{itemize}
We will denote by $\cA_{g,k(g)}$ the subset of $\cM_{g,k(g)}$ satisfying only the two first conditions (on the thin part and the systole).  Then $\cB_{\eps,\kappa,g,k(g)}$ is the intersection of $\cA_{g,k(g)}$ with the hypothesis (iii) from the spectral gap and we have:

\begin{thm}\label{thm:MMH}
Let $0 < \eps < 1$ and $g \geq 2$. Assume there exists a constant $0 \leq \kappa < 1/2$ such that $k(g) = O(g^\kappa)$. Then there exists $\beta > 0$ depending only on $\eps$ and $\kappa$ such that
$$\P_{g,k(g)}(\cB_{\eps,\kappa,g,k(g)}) = 1 - O_{\eps,\kappa}( g^{-\beta}).$$
\end{thm}

\begin{remark}\begin{itemize}\item[(1)]Theorem \ref{thm:MMH} is a combination of results of Mirzakhani \cite[Theorem 4.2]{Mi} (systole), Monk \cite[Corollary 4.4]{MonkThesis} (thin part), and a recent spectral gap result by Hide \cite[Theorem 1.3 and its proof for the rate]{Hide21} for finite area surfaces in the Weil-Petersson model of random surfaces with number of cusps $k(g)$ growing at most with rate $o(\sqrt{g})$. We also highlight the work of Shen and Wu \cite{ShenWu} where it was shown Hide's result is sharp in the sense that if $k(g)$ grows much faster than $\sqrt{g}$, random surfaces in $\cM_{g,k(g)}$ can have arbitrarily small spectral gap as $g$ grows showing that our estimate from Theorem \ref{thm:main} cannot be directly used for such surfaces with too many cusps.
\item[(2)] However, we note that Mirzakhani's result \cite[Theorem 4.2]{Mi} on the systole did not specify the quantitative dependence on the number of cusps $k(g)$ as $g \to \infty$ and the proof of the upper bound for the probability we need here was only proved for compact surfaces. We prove these missing parts in the appendix (Lemma \ref{lma:noncompactsystole}).
\end{itemize}
\end{remark}

Theorem \ref{thm:MMH} gives us quantitative estimates on the spectral gap, the systole and the rate of Benjamini-Schramm convergence valid with high probability on random surfaces of large genus. However, we still need to estimate the spectral density $N(X,I) + M(X,I)$ uniformly over the subset $\cB_{\eps,\kappa,g,k(g)}$, which Theorem \ref{thm:weyl} does not give us. For this purpose, we extend the spectral convergence result of Monk \cite{Monk} to the non-compact case, obtaining the following result.

\begin{thm} \label{thm:weylquant}
Let $I = [a,b] \subset [\frac{1}{4},+\infty)$. If $X \in \cA_{g,k(g)}$ with $k(g) = o(\sqrt{g})$, then we have
$$\frac{N(X,I) + M(X,I)}{\Vol(X)} = \frac{1}{4\pi}\int_{1/4}^\infty \1_I(\lambda) \tanh(\pi \sqrt{\lambda-1/4}) \, d\lambda + R(X,I),$$
where 
$$-C \sqrt{\frac{b+1}{\log g}} \leq R(X,I) \leq C \sqrt{\frac{b+1}{\log g}} \log \left( 2 + (b-a) \sqrt{\frac{\log g}{b+1}} \right)^{1/2}.$$
\end{thm}
Note that for Theorem \ref{thm:largegenus} we only need to know that $R(X,I)$ is bounded from below by $o(1)$ when $g \to +\infty$ uniformly over $X \in \cB_{\eps,\kappa,g,k(g)}$. However, as we explain in Section \ref{sec:quantspectralconv}, the methods of \cite{Monk} generalise entirely to the finite area case and the full theorem is of interest in itself.
The detail of how Theorems \ref{thm:MMH} and \ref{thm:weylquant} are combined to obtain Theorem \ref{thm:largegenus} is provided in Section \ref{sec:randomsurfaces}.

An interesting open problem here is to find which of the continuous or discrete spectra dominate in the large genus limit for random surfaces.  Typically the continuous part of the spectrum is expected to generically be dominant. On random surfaces, we therefore view our theorem mostly as an equidistribution result for Eisenstein series. An interesting result would be to prove the following: 

\begin{conj} For a $\P_{g,k(g)}$-random hyperbolic surface $X \in \cM_{g,k(g)}$ of large genus $g$, assuming that $k(g) = o(\sqrt{g})$ as $g \to \infty$, we have that for any compact interval $I \subset (\frac{1}{4},+\infty)$:
$$\frac{N(X,I)}{M(X,I)} = o(1)$$
when $g \to +\infty$, with high probability. 
\end{conj}

As far as we know this problem is open.

\subsection{Organisation of the article} The paper is organised as follows. In Section \ref{sec:prelim} we give the necessary background on harmonic analysis on finite volume hyperbolic surfaces and Selberg's theory we use in the spectral side of the proof. In Section \ref{sec:meanzero} we prove a version of Theorem \ref{thm:main} for mean zero observables, which is similar to the compact case but requires additional steps to handle the presence of cusps and the fact we are using the quantum mean absolute deviation instead of the variance. This is the first step of the proof of the general quantitative estimate that we prove in Section \ref{s:general} where we deal more specifically with the continuous spectrum using Maass-Selberg estimates. In Section \ref{s:spectralconv} we prove the spectral convergence (Theorem \ref{thm:weyl}). For random surfaces in Theorem \ref{thm:largegenus} we need the quantitative version of the spectral convergence (Theorem \ref{thm:weylquant}) that we prove in Section \ref{sec:quantspectralconv}. Finally in Section \ref{sec:randomsurfaces} we give the argument for the proof of Theorem \ref{thm:largegenus}. In the Appendix we extend Mirzakhani's result \cite[Theorem 4.2]{Mi} on systole for non-compact surfaces needed for Theorem \ref{thm:MMH}.
 
\section{Background} \label{sec:prelim}

In this section, we give some definitions and introduce elements of harmonic analysis on hyperbolic surfaces that we will use in the proof. For more background on the geometry and spectral theory of hyperbolic surfaces we refer to the books \cite{Bus10, Iwa02, Ber16}.

\subsection{Hyperbolic surfaces}
The hyperbolic plane is identified with the upper-half plane
$$\H = \{ z = x+iy \in \C \, | \, y > 0 \}, $$
equipped with the hyperbolic Riemannian metric
$$ ds^2 = \frac{dx^2 + dy^2}{y^2}. $$
We will denote by $d(z,z')$ the distance between two points $z,z' \in \H$.
The hyperbolic volume is given by
$$ d\mu(z) = \frac{dx \, dy}{y^2}. $$
For a measurable subset $A \subset \H$ we will use the following notation interchangeably:
$$ \mu(A) = \Vol(A) = |A|.$$

The group of isometries of $\H$ is identified with $\PSL(2,\R)$, the group of real $2 \times 2$ matrices of determinant $1$ modulo $\pm \id$, acting by M\"obius transformations
$$ \left( \begin{pmatrix} a & b \\ c & d \end{pmatrix} \in \PSL(2,\R), z \in \H\right) \mapsto \begin{pmatrix} a & b \\ c & d \end{pmatrix} \cdot z = \frac{az + b}{cz + d}.$$

\bigskip

A \emph{hyperbolic surface} can be seen as a quotient $X = \Gamma \backslash \H$ of $\H$ by a discrete subgroup $\Gamma \subset \PSL_2(\R)$.
We denote by $F$ a \emph{fundamental domain} associated to $\Gamma$. If we fix $z_0 \in \H$, an example of a fundamental domain is given by the set
$$ F = \{ z \in \H \, | \, d(z_0, z) < d(z_0, \gamma z) \text{ for any } \gamma \in \Gamma - \{\pm \id \} \}. $$
The \emph{injectivity radius} on the surface $X = \Gamma \backslash \H$ at a point $z$ is given by
$$\inj_{X} (z) = \frac12 \min\{d(z, \gamma z) \: | \: \gamma \in \Gamma-\{\id\} \}.$$
Thus  $\inj_X(z)$ gives the largest $R > 0$ such that $B_X(z,R)$ is isometric to a ball of radius $R$ in the hyperbolic plane. It is also equal to half of the length of the largest geodesic loop at $z$.

Let $g \in \PSL(2,\R)$, we define the translation operator $T_g$, such that for any function $f$ on $\H$
$$ T_g f(z) = f(g^{-1}\cdot z).$$
We will generally see a function $f$ on a hyperbolic surface $X = \Gamma \backslash \H$ as a $\Gamma$-invariant function $f : \H \to \C$,
$$ T_\gamma f(z) = f(\gamma^{-1} z) = f(z) \quad \text{for all } \gamma \in \Gamma.$$
The integral of the function on the surface is then equal to the integral of the invariant function over any fundamental domain
$$ \int_{F} f(z) \, d\mu(z).$$

\subsection{Cusps}\label{s:cusps} We will now recall some technical definitions of cusps of hyperbolic surfaces of finite area, see \cite{Iwa02} for more details. Let $X = \Gamma \setminus \H$ be a finite area hyperbolic surface. Write $\mathfrak{C}(X)$ as the set of all the cusps indexed by gothic characters $\mathfrak{b}\in \mathfrak{C}(X)$. In the Poincar\'e disc model, these are identified with elements in the boundary of $\H$ so in particular we can define $\gamma \mathfrak{b}$ for $\gamma \in \mathrm{PSL}(2,\R)$ by the action of $\mathrm{PSL}(2,\R)$ on $\H \cup \partial \H$ by M\"obius transformations. Now the \textit{stability group} of cusp $\mathfrak{b}$ is the infinite cyclic group generated by parabolic motion:
$$\Gamma_\mathfrak{b} = \{\gamma \in \Gamma : \gamma \mathfrak{b} = \mathfrak{b}\} = \langle \gamma_{\mathfrak{b}} \rangle$$
Then there exists an element $\sigma_\mathfrak{b} \in \PSL(2,\R)$ with
$$\gamma_\mathfrak{b} \mathfrak{b} = \mathfrak{b}, \quad \sigma_\mathfrak{b}^{-1} \gamma_\mathfrak{b} \sigma_{\mathfrak{b}} = \begin{pmatrix} 1 & 1 \\  \ast & 1 \end{pmatrix},$$
which is called the \textit{scaling matrix} of the cusp $\mathfrak{b}$. These notations are same as in \cite[(2.1)]{Iwa02}.

Suppose $Y > 0$ is a constant. For $Y$ sufficiently large we can find $k = k(X)$ closed loops (horocycles) $\gamma_1,\dots,\gamma_k$ in $X$ (not to be confused with $\gamma_\mathfrak{b}$, this is a slight abuse of notation as we will need these later when discussing Mirzakhani's notation of random surfaces) of equal length $1/Y$ such that we can decompose $X$ as
$$X = X(Y) \cup X_1(Y) = X(Y) \cup \bigcup_{j = 1}^k Z_j(Y),$$
where $X(Y)$ is a compact manifold with the $k$ closed horocycles $\gamma_1,\dots,\gamma_k \subset X$ as boundaries and $X_1(Y)$ is the union of disjoint topological cylinders (cusps) $Z_1(Y),\dots,Z_k(Y)$ cut along the horocycles. All the cusps cut at height $Y$ are isometric to
$$ \mathfrak C_Y =  \Gamma_\infty \backslash \left\{ z = x + iy \in \H \, | \, 0 \leq x \leq 1,y \geq Y \right\}, $$
where  $\Gamma_\infty$ is the subgroup generated by the transformation $z \mapsto z + 1$. In particular we see that
$$ \Vol ( \mathfrak C_Y) = \frac1Y.$$
For each cusp $Z_j(Y)$, using the isometry $\sigma_j : \mathfrak C_Y \to Z_j(Y)$, we can see any function $f$ of $X = \Gamma \backslash \H$ as a function $f^{(j)}(x,y) = f(\sigma_j (x,y))$ such that $f^{(j)}(x,y) =  f^{(j)}(x+1,y)$, which allows to write a Fourier series decomposition 
$$ f^{(j)} (x,y) = \sum_n f^{(j)}_n(y) e^{i n x},$$
in any cusp.

\subsection{Geodesic flow}
The tangent bundle of $\H$ can be identified with $\H \times \C$. The hyperbolic metric gives the following inner product for two tangent vectors $(z,re^{i\theta})$ and $(z,r'e^{i\theta'})$ on the tangent plane $T_z\H$
$$\langle r e^{i\theta}, r' e^{i\theta'} \rangle_z  = \frac{r \, r'}{\Im(z)^2} \cos(\theta' - \theta). $$
As a consequence, the map
$$ (z,\theta) \in \H \times \S^1 \mapsto (z, \Im(z) \, e^{i\theta}) \in \H \times \C, $$
where $\S^1 = \R/2\pi\Z$, identifies $\H \times \S^1$ with the unit tangent bundle.

The group $\PSL(2,\R)$ acts on the tangent bundle via the differential of its action on $\H$. It is well known (see for example \cite{Kat92}) that this action induces a homeomorphism between 
$\PSL(2,\R)$ and the unit tangent bundle of $\H$, such that the action of $\PSL(2,\R)$ on itself by left multiplication corresponds to the action of $\PSL(2,\R)$ on the unit tangent bundle. 

We denote by $\phi_t : \H \times \S^1 \to \H \times \S^1$ the geodesic flow associated with $\H$. The Liouville measure $d\mu \, d\theta$, where $d\theta$ is the Lebesgue measure on $\S^1$, is invariant under the action of $\phi_t$.
Via the identification $\H \times \S^1 \sim \PSL(2,\R)$, the geodesic flow is equal to the multiplication on the right by the diagonal subgroup
$$\phi_t(g) =  g \begin{pmatrix} e^{t/2} & 0 \\ 0 & e^{-t/2} \end{pmatrix}, \quad g \in G, t \in \R.$$

For a hyperbolic surface $\Gamma \backslash \H$, the unit tangent bundle is identified with $\Gamma \backslash \PSL(2,\R)$, and via this identification the geodesic flow will be given simply by
$$\phi_t(\Gamma g) =  \Gamma g \begin{pmatrix} e^{t/2} & 0 \\ 0 & e^{-t/2} \end{pmatrix}.$$

\subsection{Polar coordinates}
Let $z_0 \in \H$ be an arbitrary point. For any point $z\in\H$ different from $z_0$, there is a unique geodesic of length $r$ going from $z_0$ to $z$. Using the geodesic flow, it means that there is a unique $\theta \in \S^1$ and $r \in (0,\infty)$ such that 
$z$ is the projection of  $\phi_r(z_0,\theta)$ on the first coordinate. The change of variable
$ z \mapsto (r,\theta) $
is called \emph{polar coordinates}. The induced metric is
$$ ds^2 = dr^2 + \sinh^2 r \, d\theta^2, $$
and the hyperbolic volume in these coordinates is given by
$$ d\mu(r,\theta) = \sinh r \, dr \, d\theta. $$

\subsection{Spectrum of the Laplacian and Eisenstein series}
In the coordinates $z = x+iy$, the Laplacian $\Delta$ on $\H$ is the differential operator
$$ \Delta = -y^2 \left(\frac{\partial^2}{\partial x^2} + \frac{\partial^2}{\partial y^2} \right). $$
A fundamental property of the Laplacian is that it commutes with isometries. 
We have for any $g\in\PSL(2,\R)$,
$$ T_g \Delta = \Delta T_g.$$
The Laplacian can therefore be seen as a differential operator on any hyperbolic surface $X = \Gamma \backslash \H$. The spectrum of the Laplacian $\Delta$ on $X$ can then be decomposed into the discrete part $\lambda_0 = 0 < \lambda_1 \leq \dots$ and the absolutely continuous part $[1/4,+\infty)$, where the latter come from \textit{Eisenstein series}, which we will recall their definition from \cite[(3.11)]{Iwa02} now.

Suppose $X = \Gamma \setminus \H$ is a hyperbolic surface of has finite area with cusps $\mathfrak{C}(X)$. For each cusp $\mathfrak{b} \in \mathfrak{C}(X)$, we can associate with the Eisenstein series, which is first defined for all $s \in \C$ with $\mathrm{Re}\, s > 1$ and $z \in X$ as:
$$E_\mathfrak{b}(s,z) = \sum_{\gamma \in \Gamma_\mathfrak{b} \setminus \Gamma } (\mathrm{Im}\, \gamma_\mathfrak{b}^{-1} \gamma z)^s.$$
Here the subgroup $\Gamma_\mathfrak{b}$ and $\gamma_\mathfrak{b}$ are defined in Section \ref{s:cusps}. For each $z \in X$, the Eisenstein series have a meromorphic extension $s \mapsto E_\mathfrak{b}(s,z)$ to the whole complex plane $\C$. Then, for each $\lambda \in [1/4,+\infty)$ in the absolutely continuous part, and each cusp  $\mathfrak{b} \in \mathfrak{C}(X)$, the Eisenstein series $z\mapsto E_\mathfrak{b}(s_\lambda,z)$, $z \in X$, where $s_\lambda \in \C$ is determined by $s_\lambda(1-s_\lambda) = \lambda$, is a non-$L^2$ eigenfunctions of the Laplacian with eigenvalue $\lambda$. Note that as $\lambda \geq 1/4$, we have $s_\lambda = \frac{1}{2} + ir_\lambda$, where $r_\lambda = \pm \sqrt{\lambda - \frac{1}{4}}$. Furthermore, $\tau(r_\lambda) = \lambda$, where, recall, $\tau(r) = \frac{1}{4}+ r^2$.

\subsection{Scattering matrix and truncated Eisenstein series}
Assume the cusps $\mathfrak{b} \in \mathfrak{C}(X)$ are numbered with $j = 1, \dots, k$, and we slightly shorten the notation by identifying each cusp $\mathfrak{b}$ with exactly one $j$. For $1 \leq \ell,j \leq k$ we can expand the Eisenstein series $E_\ell$ associated with cusp $\ell$, in the $j$-th cusp $Z_j(Y)$.  Given  $s \in \C$, this expansion is of the form
$$E_{\ell}^{(j)}(s,\sigma_j(x,y)) =  \delta_{\ell j} y^s + \Phi_{\ell j}(s) y^{1-s} + \sum_{n \neq 0} f_n^{(j)}(s,y) \, e^{inx}$$
for some functions $f_n^{(j)}(s,y)$ representing the coefficients of the non-zero Fourier modes. 
The \textit{scattering matrix} is defined as the $k \times k$ matrix $\Phi(s) = (\Phi_{\ell j}(s))_{1 \leq \ell, j \leq k}$.
The determinant of $\Phi(s)$, denoted by $\phi(s)$, is called the \textit{scattering determinant}. When $\Re s = 1/2$, we have that $\Phi(s)$ is a unitary matrix. 

Given a height $Y$, we can form a truncated version $E_j^Y(s,z)$ of the Eisenstein series $E_j(s,z)$ defined for any $1 \leq \ell \leq k$ by
$$
E_j^Y(s,z ) = E_j(s,z) - \delta_{j \ell} \, (\Im \sigma_\ell^{-1} z)^s - \Phi_{j \ell}(s)  \, (\Im \sigma_\ell^{-1} z)^{1-s}
$$
if $z \in Z_\ell(Y)$, and by
$$ E_j^Y(s,z) = E_j(s,z)$$ 
if $z$ is in the compact part $X(Y)$.

For any $y > 0$ we denote by $\Pi_y^*$ the projector on functions whose zeroth Fourier mode vanish in each cusp at height higher than $y$ such that $ E_j^Y(s,z) = \Pi_Y^*E_j(s,z)$. 

\subsection{Invariant integral operators and Selberg transform}
We say that a bounded measurable kernel $K : \H \times \H \to \C$ is \textit{invariant} under the diagonal action of $\Gamma$ if for any $\gamma \in \Gamma$ we have 
$$K(\gamma \cdot z, \gamma \cdot w) = K(z,w), \quad (z,w) \in \H \times \H.$$
Assume for simplicity that $K(z,w) = 0$ whenever $d(z,w) > C$ for some constant $C>0$.
Such a kernel defines an integral operator $A$ on the surface $X$ defined for any $f \in C_c^\infty(\Gamma \backslash \H)$ by the formula, 
$$Af(z) = \int_{\H} K(z,w) f(w) \, d\mu(w) = \int_{D} \sum_{\gamma\in\Gamma} K(z,\gamma w) f(w) \, d\mu(w), \quad z \in D.$$
The function $\tilde K : \H \times \H \to \C$ given by
$$\tilde K (z,w) = \sum_{\gamma\in\Gamma} K(z,\gamma w)$$ 
is such that $\tilde K (\gamma z, \gamma' w) =\tilde K(z,w)$ for any $\gamma,\gamma' \in \Gamma$, is the Schwartz kernel of $A$.

A special case of invariant kernels is given by radial kernels. Let $\mathbf{k} : [0,+\infty) \to \C $ be a bounded measurable compactly supported function, then
$$K(z,w) = \mathbf{k}(d(z,w)), \quad (z,w) \in \H \times \H$$ 
is an invariant kernel.

For  $\mathbf{k} : [0,+\infty) \to \C $, the \textit{Selberg transform} $\cS (\mathbf{k})$ of $\mathbf{k}$ is obtained as the Fourier transform
$$ \mathcal S(\mathbf{k})(r) = \int_{-\infty}^{+\infty} e^{-iru} g(u) \, du$$
of
$$g(u) = \sqrt{2} \int_{|u|}^{+\infty} \frac{\mathbf{k}(\rho) \sinh \rho}{\sqrt{\cosh \rho - \cosh u}} \, d\rho. $$
For a function $h : \R \to \C$, the Selberg transform is inverted using the inverse Fourier transform
$$ g(u) = \frac1{2\pi}  \int_{-\infty}^{+\infty} e^{isu} h(s) \, ds $$
 and the formula
$$\mathbf{k}(\rho) = -\frac1{\sqrt{2}\pi} \int_\rho^{+\infty} \frac{g'(u)}{\sqrt{\cosh u - \cosh \rho}} \, du.$$

Eigenfunctions of the Laplacian are eigenfunctions of all operators of convolution by a radial kernel and the eigenvalues are given precisely by the Selberg transform.

\begin{prop}[\cite{Iwa02}, Theorems 1.14 and 1.16]\label{t:Stransform}
Let $X = \Gamma \backslash \H$ be a hyperbolic surface. Let $\mathbf{k} : [0,+\infty) \to \C$ be a smooth function with compact support. If $\psi_\lambda$ is an eigenfunction of the Laplacian on $X$ of eigenvalue $\lambda$, then it is an eigenfunction of the radial integral operator $A$ associated with $\mathbf{k}$. That is,
$$A \psi_\lambda(z) =  \int \mathbf{k}(d(z,w)) \psi_\lambda(w) d\mu(w) = h(r_\lambda) \psi_\lambda(z), $$
where the eigenvalue $h(r_\lambda)$ is given by the Selberg transform of the kernel $\mathbf{k}$:
$$ h(r_\lambda) = \mathcal S(\mathbf{k}) (r_\lambda),  $$
and $r_\lambda \in \R$ is defined by $\lambda = \tau(r_\lambda) = \frac{1}{4} + r_\lambda^2$.
\end{prop}

Note that this statement can be generalised to the case of $\mathbf{k}: [0,+\infty) \to \C$ measurable bounded and compactly supported by approximation and dominated convergence.

\section{Mean zero case}
\label{sec:meanzero}

We first consider the case where the test function is of mean $0$. We will consider the general case in Section \ref{s:general}. The proof of the mean zero case follows closely the proof for compact surfaces in \cite{LS17}. The main difference is the Hilbert-Schmidt norm estimate in Lemma \ref{lma:normsurface} that now takes into account the presence of cusps and an additional argument used to deal with the quantum mean absolute deviation instead of the quantum variance. 

\begin{prop}\label{l:zero} Fix $I \subset (1/4,+\infty)$ a compact interval. Then there exists $R_I > 0$ such that for all $R > R_I$ and for all hyperbolic surface $X$ with $\Vol(X) < \infty$ and for any compactly supported measurable function $a \in L^\infty(X)$ such that $ \int a(x) \, d\mu(x) = 0$, we have
\begin{align*}
 \Var_{X,I}(a) \lesssim_I C(X,I) \left( \frac{\Vol(X)}{ \,\rho(\lambda_1) R}  + \frac{e^{2R}}{\min\{1,{\inj_{X(Y_a e^{R/2})}}^2 \}} \Vol((X)_{\leq R} ) \right)^{1/2} \|a\|_\infty.
\end{align*}
where $$C(X,I) = \frac{\max\{N(X,I), k(X)\}^{1/2}}{N(X,I) + M(X,I)}, $$ $k(X)$ is the number of cusps and $\rho(\lambda_1)$ is a function of the spectral gap.
\end{prop}

The key idea proposed in \cite{LS17} is to introduce a ball averaging operator that we see as a form of wave propagation.
For any bounded measurable function $u : X \to \C$ we define
\begin{equation}\label{e:wave}
P_t \, u(z) = \frac{1}{e^{t/2}}\int_{B(z,t)} u(w) \, d\mu(w),
\end{equation}
where $B(z,t)$ is the hyperbolic ball of radius $t$ around $z$.
The operator $P_t$ is at the centre of our dynamical approach. Our goal is to show that the mean deviation for a mean-zero test function $a$
\begin{multline*}
\Var_{X,I}(a) = \frac1{N(X,I) + M(X,I)} \left( \sum_{\lambda_j \in I} \left| \langle \psi_j, a \, \psi_j \rangle \right| \right. \\ \left. + \frac1{4\pi} \int_{\tau^{-1}(I)} \left|  \sum_{\mathfrak{b}\in \mathfrak{C}(X)} \langle E_{\mathfrak{b}}(\cdot,\tfrac12 + ir), a \, E_{\mathfrak{b}}(\cdot,\tfrac12 + ir) \rangle\right| \, dr \right),$$
\end{multline*}
satisfies some invariance under the action of $P_t$ in the sense that we can formally replace $a$ in the previous expression with the time-evolved operator $\frac1T \int_0^T P_t \, a \, P_t \,dt $ and consider
$$\Var_{X,I} \left( \frac1T \int_0^T P_t \, a \, P_t \,dt \right),$$
which in turn is controlled by the dynamics of the geodesic flow when $T\to +\infty$.  In fact we just need 
$$ \Var_{X,I}(a) \lesssim \Var_{X,I} \left( \frac1T \int_0^T P_t \, a \, P_t \,dt \right) $$
where the implied constant is uniform in $T$ and $X$. For us we relate $T$ and $R$ by $T = 2R$.

\subsection{Invariance of the quantum mean absolute deviation (spectral side)} \label{sec:spectralside}

The first step is to understand the action of $P_t$ on eigenfunctions $\psi_\lambda$ of the Laplacian of eigenvalue $\lambda$.
The operator $P_t$ has the form of a radial integral operator: for $u \in L^\infty(X)$, we have
$$P_t u(z) = \int K_t(z,w) \, u(w) \, d\mu(w)$$
with radial kernel $K_t(z,w) = \mathbf{k}_t(d(z,w))$, where
$$\mathbf{k}_t(\rho) := e^{-t/2} \mathbf{1}_{\{\rho \leq t\}}.$$
By Proposition \ref{t:Stransform} we have for any function $\psi_\lambda$ such that $\Delta \psi_\lambda = \lambda \psi_\lambda$
$$P_t \psi_\lambda(z) =  \int \mathbf{k}_t(d(z,w)) \psi_\lambda(w) d\mu(w) = \cS(\mathbf{k}_t)(s_\lambda) \psi_\lambda(z), $$
where $\cS(\mathbf{k}_t)$ is the Selberg transform of the kernel $\mathbf{k}_t$ and $s_\lambda \in \C$ is defined by the equation $\lambda = \frac14 + s_\lambda^2$. 

The action of $\frac1T \int_0^T P_t \, a \, P_t \,dt $ on $\psi_\lambda$ will be understood through the following lemma, proved in \cite[Proposition 4.2]{LS17}.

\begin{lemma}\label{lma:spectralaction} Let $I \subset (\frac{1}{4},+\infty)$ be a compact interval. Then there exists a constants $C_I > 0$ and $T_I > 0$ such that for all $T > T_I$ we have
$$\inf_{r \in \tau^{-1}(I)} \frac{1}{T}\int_0^T |\cS(\mathbf{k}_t)(r)|^2 \, dt \geq C_I$$
where $\tau(r) = \frac{1}{4} + r^2$.
\end{lemma}

Given any $\lambda \in I$, an application of Lemma \ref{lma:spectralaction} to $\phi_\lambda$ gives us a bound for any $T > T_I$:
\begin{align}
 | \langle \psi_\lambda, a \, \psi_\lambda \rangle| \leq \frac{1}{C_I}  \Big| \Big\langle \psi_\lambda, \Big(\frac{1}{T}\int_0^T P_t a P_t \, dt\Big) \psi_\lambda \Big\rangle\Big|. \label{e:2}
\end{align}
This applies to both cusp forms and Eisenstein series. Thus summing and integrating the above estimate \eqref{e:2} over both the discrete and continuous spectrum, we get a formal bound for the quantum mean absolute deviation \eqref{e:qvariance}. It is not obvious that this bound is in fact finite. We will show it by establishing in Proposition \ref{prop:hsbound} that $\frac{1}{T}\int_0^T P_t a P_t \, dt$ is a Hilbert-Schmidt operator. We first have the following bound of the quantum deviation. We note that this also gives a bound for the trace of the spectral projection of the Laplacian to the interval $I$.

\begin{prop}\label{prop:qvhs}
Under the assumptions of Theorem \ref{thm:main}, there exists $T_I > 0$ such that for all $T > T_I$ and any $a \in L^\infty(X)$ with $\int_X a \, d\mu = 0$, we have
\begin{align*}
\Var_{X,I}(a) \lesssim_I C(X,I) \Big\|\frac{1}{T}\int_0^T P_t a P_t \, dt\Big\|_{\HS},
\end{align*}
where
$$C(X,I) = \frac{\max\{N(X,I), k(X)\}^{1/2}}{N(X,I) + M(X,I)} $$
and $k(X)$ is the number of cusps of $X$.
\end{prop}

Here the quantity
$$\Big\|\frac{1}{T}\int_0^T P_t a P_t \, dt\Big\|_{\HS}$$
is the Hilbert-Schmidt norm of the operator $\frac{1}{T}\int_0^T P_t a P_t \, dt$, that we bound in the next section.

\begin{proof}[Proof of Proposition \ref{prop:qvhs}] Let $\bar a = \frac{1}{\Vol(X)} \int_X a(x) \, d\mu(x)$ and assume $\bar a = 0$. Write
$$A := \frac{1}{T}\int_0^T P_t a P_t \, dt.$$
Our aim is to relate $\Var_{X,I}(a)$ to $\|A\|_{\mathrm{HS}}^2$ using the spectral theorem. By the Cauchy-Schwarz inequality, the concavity of the square root (which gives $\sqrt{a} + \sqrt{b} \leq \frac{2}{\sqrt{2}}\sqrt{a+b}$), and the bound \eqref{e:2}, we have the estimate:
\begin{align*}
& \quad \sum_{\lambda_j \in I} | \langle \psi_j, a \, \psi_j \rangle| + \int_{\tau^{-1}(I)} \left|  \sum_{\mathfrak{b}\in \mathfrak{C}(X)} \langle E_{\mathfrak{b}}(\cdot,\tfrac12 + ir), a \, E_{\mathfrak{b}}(\cdot,\tfrac12 + ir) \rangle \right| \, dr  \\
&  \leq N(X,I)^{1/2}\Big(\sum_{\lambda_j \in I} | \langle \psi_j, a \, \psi_j \rangle|^2\Big)^{1/2} \\
& \qquad + (|\tau^{-1}(I)| |\mathfrak{C}(X)|)^{1/2}\Big(\int_{\tau^{-1}(I)}   \sum_{\mathfrak{b}\in \mathfrak{C}(X)} \left| \langle E_{\mathfrak{b}}(\cdot,\tfrac12 + ir), a \, E_{\mathfrak{b}}(\cdot,\tfrac12 + ir) \rangle \right|^2 \, dr\Big)^{1/2}  \\
&  \lesssim \max\{N(X,I),|\tau^{-1}(I)| |\mathfrak{C}(X)|\}^{1/2}\Big(\sum_{\lambda_j \in I} | \langle \psi_j, a \, \psi_j \rangle|^2\\
& \qquad + \int_{\tau^{-1}(I)}  \sum_{\mathfrak{b}\in \mathfrak{C}(X)} \left| \langle E_{\mathfrak{b}}(\cdot,\tfrac12 + ir), a \, E_{\mathfrak{b}}(\cdot,\tfrac12 + ir) \rangle \right|^2 \, dr\Big)^{1/2}  \\
& \lesssim_I (\max\{N(X,I),|\tau^{-1}(I)| |\mathfrak{C}(X)|\})^{1/2}\Big(\sum_{\lambda_j \in I} \Big| \Big\langle \psi_j, A \psi_j \Big\rangle\Big|^2 \\
&\qquad+  \frac{1}{4\pi}\int_{\tau^{-1}(I)}   \sum_{\mathfrak{b}\in \mathfrak{C}(X)} \left| \Big\langle E_{\mathfrak{b}}(\cdot,\tfrac12 + ir), A \, E_{\mathfrak{b}}(\cdot,\tfrac12 + ir) \Big\rangle\right|^2 \, dr\Big)^{1/2}\\
& \lesssim_I   \max\{N(X,I),|\tau^{-1}(I)| |\mathfrak{C}(X)|\}^{1/2} \|A\|_{\HS}.
\end{align*}
For the last inequality we use the spectral theorem (See \cite[Theorem 7.3 and Theorem 7.4]{Iwa02}) for the kernel of the operator $A$, which gives
\begin{align*}\sum_{\lambda_j \in I} | \langle \psi_j,A \psi_j \rangle |^2 +  \frac{1}{4\pi}\int_{\tau^{-1}(I)}   \sum_{\mathfrak{b}\in \mathfrak{C}(X)} | \langle E_{\mathfrak{b}}(\cdot,\tfrac12 + ir), A E_{\mathfrak{b}}(\cdot,\tfrac12 + ir) \rangle |^2 \, dr \leq \| A \|_{\HS}^2.
\end{align*}
\end{proof}

\subsection{Bounding the Hilbert-Schmidt norm} \label{sec:geometricside}

Write $R_I = 2T_I$, where $T_I > 0$ is from Proposition \ref{prop:qvhs}. Let $R > R_I$ and write $T = R/2$ so $T > T_I$. Then by Proposition \ref{prop:qvhs}, to prove Theorem \ref{thm:main} we will need to bound 
$$\Big\|\frac{1}{T}\int_0^T P_t a P_t \, dt\Big\|_{\HS}.$$
Since the test function $a \in L^\infty(X)$ has compact support in $X$, we can choose $Y = Y_a > 0$ large enough such that
$$ \spt a \subset X(Y) $$
where
$$ X(Y) = X \setminus \bigcup_{\mathfrak{b}} X_{\mathfrak{b}} (Y),$$
and $X_{\mathfrak{b}} (Y)$ is the cuspidal zone associated with $\mathfrak{b}$. This means that the support of $a$ does not go beyond height $Y$ into the cusps.

We can prove that $\frac{1}{T}\int_0^T P_t a P_t \, dt$ is Hilbert-Schmidt and has the following quantitative bound that we apply with $T = 2R$:

\begin{prop}[Geometric bound]\label{prop:hsbound} For every $a \in L^\infty(X)$ compactly supported and every $T > 0$ the operator $\frac{1}{T}\int_0^T P_t a P_t \, dt$ is Hilbert-Schmidt with norm
\begin{align*}\Big\|\frac{1}{T}\int_0^T P_t a P_t \, dt\Big\|_{\HS}^2   \lesssim \frac{\|a\|_2^2}{T\rho(\lambda_1)} +  \frac{e^{4T}}{\min\{1,{\inj_{X(Y_a e^T)}}^2 \}}  \Vol((X)_{\leq 2T}) \|a\|_\infty^2,
\end{align*}
for $Y_a >0$ large enough such that $\spt a  \subset X(Y_a)$.
\end{prop}

We will work with a fundamental domain $F$ of $X$ that we decompose such that:
$$ F(Y) = F \setminus \bigcup_{\mathfrak{b}} F_{\mathfrak{b}} (Y),$$
and $F_{\mathfrak{b}} (Y)$ represent the cuspidal zone associated with the cusp $\mathfrak{b}$. Having fixed $T > 0$ and $a \in L^\infty(X)$, we write for $(z,w) \in F \times F$
$$\mathbf{K}_T(z,w) = \Big[\frac{1}{T}\int_0^T P_t a P_t \, dt \Big](z,w) = \frac{1}{T}\int_0^T [P_t a P_t](z,w) \, dt$$
where we use the bracket notation $[A]$ for the kernel of an integral operator $A$. Then we have
$$\left\|\frac{1}{T}\int_0^T P_t a P_t \, dt\right\|_{\HS}^2 = \int_F \int_F |\mathbf{K}_T(z,w)|^2 d\mu(z) d\mu(w).$$
Now the kernel $\mathbf{K}_T : X \times X \to \R$ on $X$ can be represented as an invariant kernel $K_T : \H \times \H \to \R$ under the diagonal action of $\Gamma$ on $\H \times \H$ as follows:
$$\mathbf{K}_T(z,w) = \sum_{\gamma \in \Gamma} K_T(z,\gamma \cdot w)$$
for any $(z,w) \in F \times F$. In our case, seeing $a$ as a $\Gamma$-invariant function on $\H$, we can write in the above
$$K_T(z,w) = \frac{1}{T}\int_0^T  e^{-t}  \int\limits_{B(z,t) \cap B(w,t)} a(x) \, d\mu(x) \, dt.$$ 
Thus in particular, we have that the Hilbert-Schmidt norm can be written as
$$\left\|\frac{1}{T}\int_0^T P_t a P_t \, dt\right\|_{\HS}^2 =  \int_F \int_F \Big| \sum_{\gamma \in \Gamma} K_T(z, \gamma \cdot w) \Big|^2 d\mu(z) d\mu(w).$$
Hence to prove Proposition \ref{prop:hsbound}, we need to estimate Hilbert-Schmidt norms of integral operators with invariant kernels, which we do in the following lemma. 

Recall $(F)_{\leq 2T}$ denotes the points in the fundamental domain $F$ with radius of injectivity less than $2T$:
$$(F)_{\leq R} = \{ z \in F  :  \inj_{X}(z) \leq 2T \},$$
and we denote by $(F)_{> 2T}$ the complement of this set in $F$. We write $\H(Y) = \Gamma \cdot F(Y)$ for all the images of the compact part of $F$ by the action of $\Gamma$.

\begin{lemma}\label{lma:normsurface}
Let $A$ be an integral operator on $X$ such that
\begin{equation}\label{e:HS}
\| A \|_{\HS}^2 = \int_F \int_F \Big| \sum_{\gamma \in \Gamma} K(z, \gamma \cdot w) \Big|^2 d\mu(z) d\mu(w).
\end{equation}
for a kernel $K : \H \times \H \to \R$ invariant under the diagonal action of $\Gamma$ on $\H \times \H$. Fix $T, Y > 0$. We assume that the kernel $K$ satisfies $K(z,w) \neq 0$ when 
\begin{itemize}
\item[(1)] $d(z,w) \leq 2T$ and $z,w \in \H(Y)$, or 
\item[(2)] there exists $x \in \H(Y)$ such that $d(z,x) \leq T$ and $d(w,x) \leq T$.
\end{itemize}
Then we have:
$$
\| A \|_{\HS}^2 \leq \int_F \int_\H  |K(z, w) |^2 d\mu(z) d\mu(w) 
+ \frac{e^{4T}}{\min\{1,{\inj_{X(Y e^T)}}^2 \}}\Vol\left( (F)_{\leq 2T}\right) \sup_{(z,w) \in F\times\H} |K(z,w)|^2.
$$
\end{lemma}

\begin{proof}
This is a more general version of Lemma 5.1 of \cite{LS17} on Hilbert-Schmidt norm estimates in terms of the injectivity radius, that allows us to treat the case when $X$ is not compact. Write $R = 2T$. Assume first that the case (1) holds for the support of $K$, that is, $K(z,w) \neq 0$ when $d(z,w) \leq 2T$ and $z,w \in \H(Y)$. We split the integral \eqref{e:HS} into two parts over points with small and large radius of injectivity, and use that in the first part, the sum over $\Gamma$ is reduced to one term.
\begin{align*}
	\| A \|_{\HS}^2 = 
	\int_{(F)_{> R}} \int_F \sum_{\gamma \in \Gamma} | K(z, \gamma \cdot w) |^2 d\mu(z) d\mu(w) + \int_{(F)_{\leq R}} \int_F | \sum_{\gamma \in \Gamma} K(z, \gamma \cdot w) |^2 d\mu(z) d\mu(w).
\end{align*}
We get using the Cauchy-Schwarz inequality that
 $$| \sum_{\gamma \in \Gamma} K(z, \gamma \cdot w) |^2 \leq N_\Gamma(R;z,w) \sum_{\gamma \in \Gamma} |K(z, \gamma \cdot w) |^2,$$
with the lattice counting parameter
$$N_\Gamma(R;z,w) = \sharp \{\gamma \in \Gamma : d(z,\gamma w) \leq R\}.$$
Now since $z, w \in \H(Y)$, for any $\gamma \in \Gamma$ such that $d(z,\gamma w) \leq R$, we have $$B(\gamma w, \inj_{X(Y)}) \subset B(z, R+ \inj_{X(Y)})$$ and if $\gamma' \in \Gamma - \{\gamma\}$ then $\gamma' w \notin B(\gamma w, \inj_{X(Y)})$. We deduce that the number of lattice points $N_\Gamma(R;z,w)$ is bounded by the number of balls of radius $\inj_{X(Y)}$ that one can fit in a ball of radius $R + \inj_{X(Y)}$, that is
$$|N_\Gamma(R;z,w)| \leq  \frac{\cosh(R + \inj_{X(Y)}) - 1 }{\cosh(\inj_{X(Y)}) - 1} 
\lesssim \frac{e^R}{\min\{1,{\inj_{X(Y)}}^2 \}}$$
where the implied constant is universal.
The rest is similar to the proof of Lemma 5.1 in \cite{LS17}: we have
$$ \| A \|_{\HS}^2 \lesssim 
	\int_{F} \int_\H | K(z, w) |^2 d\mu(z) d\mu(w)
	+  \frac{e^R}{\min\{1,{\inj_{X(Y)}}^2 \}} \int_{(F)_{\leq R}} \int_\H |K(z, w) |^2 d\mu(z) d\mu(w).
$$
The second term on the right-hand side is bounded by 
$$ \frac{e^R}{\min\{1,{\inj_{X(Y)}}^2 \}} \Vol(B(R))\Vol\left( (F)_{\leq R}\right) \sup_{(z,w) \in F\times\H} |K(z,w)|^2,$$
and $\Vol(B(R)) \lesssim e^R$, which concludes the proof in the case (1). For the case (2), that is, $K(z,w) \neq 0$ when there exists $x \in \H(Y)$ such that $d(z,x) \leq T$ and $d(w,x) \leq T$. We now note that if such $x$ exist, then the height of $z$ and $w$ is bounded from above by $T e^{T}$. Therefore, we need to adjust the estimate with the division by $\inj(Y e^{T}) \geq \inj(Y)$ as $e^T \geq 1$.
\end{proof}

We are interested in the invariant kernel
$$K_T(z,w) = \frac{1}{T}\int_0^T  e^{-t} \int\limits_{B(z,t) \cap B(w,t)} a(x) \, d\mu(x) \, dt$$ 
associated with $\frac{1}{T}\int_0^T P_t a P_t \, dt$. We see that $K(z,w) = 0$ whenever $d(z,w) \geq 2T$ so Lemma \ref{lma:normsurface} can be applied with $R = 2T$. Hence in order to prove Proposition \ref{prop:hsbound} we are left with proving $L^2$ and $L^\infty$ estimates for our invariant kernel. 

The $L^\infty$ bound is straightforward, we have
\begin{equation}\label{lma:linftybound}
\sup_{(z,w) \in F\times\H} |K_T(z,w)|^2 \lesssim \|a\|_\infty^2,
\end{equation}
since $\Vol(B(z,t) \cap B(w,t)) \lesssim \Vol(B(t)) \lesssim e^t$ for all $(z,w) \in F \times \H$.

The $L^2$ bound is at the core of our analysis.

\begin{lemma}\label{lma:l2bound} If $T > 0$, then we have
$$\int_F \int_\H  |K_T(z,w)|^2 d\mu(z) d\mu(w) \lesssim \frac{\|a\|_2^2}{T \rho(\beta)^2}.$$
\end{lemma}

The proof of this follows from a quantitative ergodic theorem by Nevo published in \cite{Nev98} (see also \cite{LS17} for more explanations on the application to our setting).

Let $(\mathcal{X},\nu)$ be a probability space, and $G$ a group equipped with its left-invariant Haar measure $dg$, and a measure-preserving action on $\mathcal{X}$. 
For a collection of measurable sets $A_t \subset G$ we define the averaging operators 
$$\pi_\mathcal{X}(A_t)f(x) = \frac{1}{|A_t|} \int_{A_t} \, f(g^{-1} x) \, dg, \quad f \in L^2(\mathcal{X}), \quad x \in \mathcal{X}.$$
This result was proved by Nevo in \cite{Nev98} and stated in this form by Gorodnik and Nevo in \cite[Theorem 4.1]{GN15}:

\begin{theorem}\label{thm:nevo}
If $G$ is a connected simple Lie group equipped with a measure-preserving action on the probability space $(\mathcal{X},\nu)$ that has a spectral gap, then there exist $C, \theta > 0$ such that for any family $A_t \subset G$, $t \geq 0$, of measurable sets of positive measure, we have 
$$\left\| \pi_\mathcal{X}(A_t) f - \int_\mathcal{X} f \, d\mu\right\|_{L^2(\mathcal{X},\nu)} \leq C \, |A_t|^{-\theta} \, \|f\|_{L^2(\mathcal{X},\nu)}$$
for any $f \in L^2(\mathcal{X},\nu)$, where we denote by $|A_t|$ the measure of the set $A_t$. The constant $C$ depends only on $G$ and $\theta$ depends only on the spectral gap.
\end{theorem}

Theorem \ref{thm:nevo} applies in particular when $G = \PSL(2,\R)$ and $\mathcal{X} = \Gamma \backslash \PSL(2,\R)$ for $\Gamma$ co-finite as in our setting. The important point is then that the spectral gap of the Laplacian implies that the action of $G$ on $\mathcal{X}$ has a spectral gap, and that $\theta$ depends on the spectral gap of the Laplacian. Note that we could also write this theorem for any measurable set but we want to emphasise that we see this as an ergodic theorem (or equidistribution theorem), with the idea that $|A_t|$ is increasing with $t$.

In order to use Theorem \ref{thm:nevo}, we need to use a change of variable and lift the kernels $K_T$ to $\SL(2,\R)$.

\begin{proof}[Proof of Lemma \ref{lma:l2bound}]The proof is identical to the one in Section 7 of \cite{LS17}. We briefly reproduce its main steps for the convenience of the reader. We identify $\PSL(2,\R)$ with the unit tangent bundle $\{ (z,\theta) \in \H\times \S^1 \}$ of $\H$ (see Section \ref{sec:prelim}). We define $A_t(r) \subset \PSL(2,\R)$ to be a set such that $A_t(r)^{-1} \cdot (z,\theta)$ is the lift in the unit tangent bundle of two balls of radius $t$ with centres given by the projections $z_1$ and $z_2$ onto $X$ of the points $\phi_{-r/2}(z,\theta)$ and $\phi_{r/2}(z,\theta)$ of the unit tangent bundle. Here $\phi_t$ is the geodesic flow on the unit tangent bundle of $X$.

Here we recall some notation from  \cite[Lemma 7.1]{LS17}. Let $F$ be the fundamental domain associated with $X$. Writing $B_{2T} = \{(z_1,z_2) \in F \times \H : d(z_1,z_2) < 2T\}$, define a mapping $\Phi : B_{2T} \to F \times \S^1 \times (0,2T)$ by
$$\Phi(z_1,z_2) = (m(z_1,z_2),\theta(z_1,z_2),d(z_1,z_2)).$$
Here $m(z_1,z_2)$ is the middle point of the geodesic between $z_1$ and $z_2$, the vector $\theta(z_1,z_2)$ is the direction of the unit vector at $m(z_1,z_2)$ tangent to the geodesic between $z_1$ and $z_2$, and $d(z_1,z_2)$ is the geodesic distance between $z_1$ and $z_2$. Then \cite[Lemma 7.1]{LS17} states that for any $f : \H \times \S^1 \times [0,\infty) \to \C$ that satisfies
$$f(\gamma \cdot (z,\theta), r) = f(z,\theta,r) \quad \forall \gamma \in \Gamma,$$
there is a change of variable:
$$\iint_{B_{2T}} f(\Phi(z_1,z_2)) \, d\mu(z_1)\, d\mu(z_2)= \int_{0}^{2T} \sinh(r) \int_F \int_{\S^1} f(z,\theta,r)\, d\theta \, d\mu(z) \, dr.$$
We will use the following function:
$$f(z,\theta,r) := \left|\frac{1}{T} \int_{r/2}^T e^{-t} |A_t(r)| \, \pi(A_t(r)) a (z,\theta) \, dt\right|^2, \quad (z,\theta,r) \in \H \times \S^1 \times [0,\infty),$$
where
$$ \pi(A_t(r))a(z,\theta) = \frac{1}{|A_t(r)|} \int_{A_t(r)} \, f(g^{-1} \cdot (z,\theta)) \, dg.$$
This function satisfies for all $(z_1,z_2) \in B_{2T}$ that
$$f(\Phi(z_1,z_2)) = \left|\frac{1}{T}\int_{d(z_1,z_2)/2}^T  e^{-t} \int\limits_{A_t(d(z_1,z_2))} a(g^{-1} \cdot (m(z_1,z_2),\theta(z_1,z_2))) \, dg \, dt\right|^2.$$
The set $\{g^{-1} \cdot (m(z_1,z_2),\theta(z_1,z_2)) \, | \, g \in A_t(d(z_1,z_2)) \}$ is the lift to the unit tangent bundle of $B(z_1,t) \cap B(z_2,t)$ by definition of $A_t(d(z_1,z_2))$, and the Haar measure $dg$ descends onto the hyperbolic area $d\mu$ on $B(z_1,t) \cap B(z_2,t)$. As $B(z_1,t) \cap B(z_2,t) = \emptyset$ if $t < d(z_1,z_2) / 2$, we thus have
\begin{align*} f(\Phi(z_1,z_2)) = \Big|\frac{1}{T}\int_0^T  e^{-t} \int\limits_{B(z_1,t) \cap B(z_2,t)} a(x) \, d\mu(x) \, dt\Big|^2 = |K_T(z_1,z_2)|^2.
\end{align*}
Therefore, as $K_T(z_1,z_2) = 0$ if $d(z_1,z_2) \geq 2T$, we have by the change of variable that
\begin{align*}& \quad \,\, \int_{F} \int_\H | K_T(z_1, z_2) |^2 d\mu(z_2) d\mu(z_2) \\
&=  \int_0^{2T} \sinh r \int_F \int_{\S^1} \left|\frac{1}{T} \int_{r/2}^T e^{-t} |A_t(r)| \, \pi(A_t(r)) a (z,\theta) \, dt\right|^2  \, d\theta \, d\mu(z) \,dr.\end{align*}
Then Minkowski's integral inequality yields
	\begin{align*}
& \int_0^{2T} \sinh r \int_F \int_{\S^1} \left|\frac{1}{T} \int_{r/2}^T e^{-t} |A_t(s)| \, \pi(A_t(r)) a (z,\theta) \, dt\right|^2  \, d\theta \, d\mu(z) \,dr,\\
		&\quad \leq \int_0^{2T} \sinh r \left( \frac{1}{T} \int_{r/2}^T e^{-t} |A_t(r)|
		\left\| \pi(A_t(r)) a \right\|_{L^2(F\times\S^1)} \, dt\right)^2\,dr.
	\end{align*}
By Theorem \ref{thm:nevo} with $G = \PSL(2,\R)$ and $\mathcal{X} = \Gamma \backslash \PSL(2,\R)$, there is a constant $\rho(\beta) > 0$ depending only on the spectral gap $\beta$ of the Laplacian and a constant $C > 0$ that only depends on $G = \PSL(2,\R)$ by Theorem \ref{thm:nevo} such that
$$\left\| \pi(A_t(r)) a \right\|_{L^2(F\times\S^1)} \leq C \, |A_t(r)|^{-\rho(\beta)} \, \|a\|_2.$$
Hence the previous integral is bounded by
	\begin{align} \label{eq:finalone}
		\int_0^{2T} \sinh r \left( \frac{1}{T} \int_{r/2}^T e^{-t} |A_t(s)|^{1-\rho(\beta)}
		\left\| a \right\|_2 \, dt\right)^2\,dr.
	\end{align}
	Note that as $A_t(s)$ is given by the lift of intersection $B_1 \cap B_2$ of two balls $B_1,B_2 \subset \H$ of radius $t$ such that their centres is at a distance $s$ from each other and that $B_1 \cap B_2$ is contained in a ball $B(z,\rho)$ of radius $\rho > 0$ that satisfies $\cosh \rho = \frac{\cosh t}{\cosh(r/2)}$ by the hyperbolic Pythagoras theorem. As a hyperbolic disc, the area of $B(z,\rho)$ is $4\pi \sinh^2(\rho/2)$, which is therefore bounded above by $Ce^{t-r/2}$, for some uniform constant $C > 0$. A picture and more details on this can be found from Figure 2 of \cite{LS17}. Thus we know that for some uniform constant
	$$|A_{t}(r)| \lesssim e^{t-r/2}.$$
	We can thus bound \eqref{eq:finalone} with a uniform constant times
	\begin{align*}
\int_0^{2T} \sinh r \left( \frac{1}{T} \int_{r/2}^T e^{-s/2} e^{-\rho(\beta)(t-r/2)}
		\left\| a \right\|_2 \, dt\right)^2\,dr \lesssim \frac{1}{T^2}  \int_0^{2T} \frac{\|a\|_2^2}{\rho(\beta)^2} \,dr  \lesssim \frac{\|a\|_2^2}{T \rho(\beta)^2},
	\end{align*}
	so the proof of the claim is complete
	\end{proof}	

Combining \eqref{lma:linftybound} and Lemma \ref{lma:l2bound} with Lemma \ref{lma:normsurface} we thus proved the desired bound claimed in Proposition \ref{prop:hsbound}. Together with Proposition \ref{prop:qvhs}, this completes the proof of Proposition \ref{l:zero}.

\section{General case: Proof of Theorem \ref{thm:main} and Theorem \ref{thm:mainBS}} \label{s:general}

In this section, we treat the general case of observables with non-zero mean, and prove Theorem \ref{thm:main} and Theorem \ref{thm:mainBS}. 

\subsection{Proof of Theorem \ref{thm:main}}
If $a$ is a test function that does not have mean $0$, i.e.
$$\overline{a} := \frac{1}{\Vol(X)} \int a(z) \, dz \neq 0$$
we fix an arbitrary $Y \geq Y_a$ where $Y_a$ is defined as the smallest height such that the support of $a$ is in $X(Y_a)$. We then define
$$b(z) := a(z) - \overline{a} \chi(z) , \quad z \in X$$
where 
$$\chi(z) = \begin{cases}
\frac{\Vol(X)}{\Vol(X(Y))}, & \text{if } z \in X(Y);\\
0, &  \text{otherwise}.
\end{cases}
$$
This idea to use such a symbol is similar to what is done in \cite{BZ16}, albeit simplified by the fact we do not need $b$ smooth, as our proof works for $L^\infty$ test functions.

By this choice of $\chi$ we have that
$$\int_X b(z) \, dz = \int_X a(z) \, dz -\overline{a}   \int \chi(z) \, dz = 0.$$
Write
$$\widetilde \Var_{X,I}(a) = (N(X,I)+M(X,I)) \Var_{X,I}(a).$$
Then we have
$$\widetilde \Var_{X,I}(a) \lesssim  \widetilde \Var_{X,I}(b) +  \overline{a} \left( \sum_{\lambda_j \in I} \left| \int (\chi(z) - 1) |\psi_j(z)|^2 \, dz \right|  +  \cE_{X,I} \right), $$
where
$$\cE_{X,I} = \int_{\tau^{-1}(I)} \left|\int_{X} \chi(z) \sum_{j = 1}^k |E_j(r,z)|^2 \, dz + \frac{\phi_X'(\frac12 + ir)}{\phi_X(\frac12+ ir)} \right| \, dr.$$
By Proposition \ref{l:zero} we have 
$$ \Var_{X,I}(b) \lesssim_I C(X,I) \left( \frac{1}{ \,\rho(\lambda_1) R} \|b\|_2^2 + \frac{e^{2R}}{\min\{1,{\inj_{X(Y e^{R/2})}}^2 \}} \Vol((X)_{\leq R} ) \|b\|_\infty^2 \right)^{1/2},$$
and using that $a$ is supported inside $X(Y)$ we can compute that
$ \| b \|_\infty \leq 2 \| a \|_\infty$, and $\| b \|_2 \leq 2 \| a \|_2$. Moreover, $z \mapsto \chi(z) - 1$ is of mean $0$ and we have
\begin{align*}
&\frac1{N(X,I) + M(X,I)} \sum_{r_j \in I} \left| \int (\chi(z) - 1) |\psi_j(z)|^2 \, dz \right|^2 \leq \Var_{X,I} (\chi - 1)\\
	&\qquad \lesssim_I C(X,I) \left( \frac{1}{ \,\rho(\lambda_1) R} \|\chi - 1\|_2^2 + \frac{e^{2R}}{\min\{1,{\inj_{X(Y e^{R/2})}}^2 \}} \Vol((X)_{\leq R} ) \|\chi - 1\|_\infty^2\right)^{1/2}.
\end{align*}
Note now that by Cauchy-Schwartz inequality 
$$\overline{a}^2 \leq \frac{1}{|X|} \|a\|_2^2 \leq \|a\|_\infty^2$$ 
and we also have by the definition of $\chi$ that $\| \chi - 1 \|_\infty \leq 1$. Hence
$$\overline{a}^2 \| \chi - 1 \|_2^2 \leq \| a \|_2^2\quad \text{and} \quad \overline{a}^2 \| \chi - 1 \|_\infty^2 \leq \|a\|_\infty^2,$$ 
so in the end
\begin{align*}
\Var_{X,I}(a) &\lesssim_I C(X,I) \left(\frac{1}{ \,\rho(\lambda_1) R} \|a\|_2^2 + \frac{e^{2R}}{\min\{1,{\inj_{X(Ye^{R/2})}}^2 \}} \Vol((X)_{\leq R} ) \|a\|_\infty^2\right)^{1/2} \\
&\quad  + \frac1{N(X,I) + M(X,I)} \overline{a} \cE_{X,I}.
\end{align*}
Therefore just need to estimate $\overline{a} \cE_{X,I}$. 
We have
\begin{align*}
\overline{a} \cE_{X,I} &\leq \overline{a} \frac{\Vol(X)}{\Vol(X(Y))} \int_{\tau^{-1}(I)} \left|\int_{X(Y)} \sum_{j = 1}^k |E_j(r,z)|^2 \, dz +  \frac{\Vol(X(Y))}{\Vol(X)}  \frac{\phi_X'(\frac12 + ir)}{\phi_X(\frac12+ ir)} \right| \, dr\\
&\leq \|a\|_\infty \int_{\tau^{-1}(I)} \left|\int_{X(Y)} \sum_{j = 1}^k |E_j(r,z)|^2 \, dz +  \frac{\Vol(X(Y))}{\Vol(X)}  \frac{\phi_X'(\frac12 + ir)}{\phi_X(\frac12+ ir)} \right| \, dr
\end{align*}
We can use the Maass-Selberg relations (see \cite[Section 2]{Sarnak1981}).
\begin{lemma}[Maass-Selberg relations]\label{lma:maassselberg}
Suppose $X$ has $k$ cusps. $s = \frac{1}{2} + ir$. Then for any $y$ we have
\begin{align*}\int\limits_{X(Y)} \sum_{j=1}^k |E_j(s,z)|^2 \, dz  = &\,\, 2 k \log Y  - \frac{\phi_X'(s)}{\phi_X(s)} + \Tr \Big(\frac{Y^{2i r} \Phi_X^*(s) - Y^{-2i r} \Phi_X(s)}{2i r}\Big) \\
& + \int\limits_{X \setminus X(Y)}\sum_{j=1}^k |\Pi_Y^* E_j(s,z)|^2 \, dz
\end{align*}
\end{lemma}

As the scattering matrix $\Phi_X(s)$ is unitary when $\Re(s) = 1/2$, we have in this case $|\Tr \Phi_X(s)| \leq k$. Hence by the linearity of the trace
$$\left| \Tr \Big(\frac{Y^{2i r} \Phi_X^*(s) - Y^{-2i r} \Phi_X(s)}{2i r}\Big) \right| = \frac{|\sin(2r \log Y)|}{r} |\Tr \Phi_X(s)| \leq \frac{|\sin(2r \log Y)|}{r}k,$$
using that $\Tr \Phi_X(s)^* = \Tr \Phi_X(s)$. Moreover, as $s = \frac{1}{2} + ir$, we have for all $z = x+ iy \in X \setminus X(Y)$ that
$$|\Pi_Y^* E_j(s,z)| \lesssim_I e^{-2\pi y}$$
where the implied constant depends on the spectral interval $I$ (see for example Iwaniec \cite[(6.20)]{Iwa02}).

We thus have for $Y \geq Y_a$
\begin{align*}
\overline{a} \cE_{X,I} 
& \lesssim_I  \|a\|_\infty \left( 2k \log Y + k^2 e^{-4\pi Y} \right) + \|a\|_\infty \left( 1 - \frac{\Vol(X(Y))}{\Vol(X)}  \right) \int_{\tau^{-1}(I)} \left|\frac{\phi_X'(\frac12 + ir)}{\phi_X(\frac12+ ir)} \right| \, dr
\end{align*}
where we used that $\Vol(X \setminus X(Y)) \leq k$. We then notice that
\begin{align*}
1 - \frac{\Vol(X(Y))}{\Vol(X)} &= \frac{ \Vol(X \setminus X(Y))}{\Vol(X)} \leq \frac{k}{\Vol(X)}.
\end{align*}

Finally, the proof of Theorem \ref{thm:main} is concluded if we can establish the following geometric bound for the scattering determinant:

\begin{lemma} Recall $$M(X,I) := \frac{1}{4\pi} \int_{\tau^{-1}(I) }\frac{-\phi_X'}{\phi_X}\left(\tfrac{1}{2}+ ir \right)\, dr$$
Suppose $X$ has $k$ cusps.  Then
$$\int_{\tau^{-1}(I)} \left|\frac{\phi_X'(\frac12 + ir)}{\phi_X(\frac12+ ir)} \right| \, dr \lesssim_I M(X,I) + k \log \Vol(X) + \Vol(X)$$
\end{lemma}

The proof of this lemma follows by applying the following proposition with $f = \1_{\tau^{-1}(I)}$ on the scattering determinants and continuous spectrum:

\begin{prop} \label{lma:comparescattering} Let $X$ be a hyperbolic surface of finite volume and $k$ cusps. Then for all $f \in L^1(\R)$ with $f \geq 0$, we have
$$\int f(r) \Big|\frac{\phi_X'(\frac12 + ir)}{\phi_X(\frac12+ ir)} \Big| \, dr  - \Big| \int f(r) \frac{-\phi_X'(\frac12 + ir)}{\phi_X(\frac12+ ir)} dr\Big| \lesssim \Big( k|\log \Vol(X)| + \Vol(X)\Big) \|f\|_1.$$
where the implied constant does not depend on $X$.
\end{prop}

\begin{proof}
The proof uses crucially the formula
\begin{align}\label{eq:scatteringformula}-\frac{\phi_X'(\frac12 + ir)}{\phi_X(\frac12+ ir)} = 2 \log b_1 + \sum_{\rho} \frac{2 \mathrm{Re} \rho - 1}{(1/2 - \mathrm{Re} \rho)^2 + (r - \mathrm{Im} \rho)^2},\end{align}
\cite[(11.9)]{Iwa02} where $\rho$ runs over all the poles of $\phi_X(s)$ with multiplicities and $b_1 = b_1(X) > 0$ is a constant.

The constant $b_1 = b_1(X)$ has a meaning in terms of the geometry of $X$. This was explained in \cite{Iwa02} and we will summarise it here as we the claim follows from a geometric bound for it. First of all, by the formula \cite[(3.21)]{Iwa02} at every $s \in \C$ with $\mathrm{Re} (s) > 1$, for two cusps $\ell,j$, the entry $(\Phi_{X}(s))_{\ell j}$ of the scattering matrix has a Dirichlet series representation
$$(\Phi_X(s))_{\ell j} = \pi^{1/2} \frac{\Gamma(s-1/2)}{\Gamma(s)} \sum_{c} c^{-2s} S_{\ell j}(0,0;c),$$
where $\Gamma(\cdot)$ is the Gamma-function, $S_{\ell j}(0,0;c)$ is the Kloosterman sum \cite[(2.23)]{Iwa02} and the sum is over real numbers $c > 0$ from the set (notation from \cite[(2.22)]{Iwa02}):
$$\cC_{\ell j} = \Big\{ c > 0 : \begin{pmatrix} \ast & \ast \\ c & \ast \end{pmatrix} \in \sigma_{\ell}^{-1} \Gamma \sigma_j\Big\},$$
where $\sigma_\ell$ and $\sigma_j$ are the scaling matrices associated to cusps $\ell$ and $j$, recall Section \ref{s:cusps} for definitions of these. This also then implies, by the definition of the determinant, as noted in \cite[Page 160]{Iwa02} that the scattering determinant $\phi_X(s) = \det \Phi_X(s)$ has the Dirichlet series representation
\begin{align}\label{eq:dirichlet}\phi_X(s) = \Big(\sqrt{\pi} \frac{\Gamma(s-1/2)}{\Gamma(s)}  \Big)^{k} \sum_{n = 1}^\infty a_n b_n^{-2s},\end{align}
where $a_1 \neq 0$ and $0 < b_1 < b_2 < \dots < b_n \to \infty$. Here the terms $b_n$ are given by the length $k$ products of possible combinations of the real numbers $c$ appearing in the Dirichlet series representation of each $(\Phi_X(s))_{\ell j}$. The term $a_n$ is then the corresponding coefficient $S_{\ell j}(0,0;c)$ containing also the sign information. Thus the element $b_1$ is the \textit{smallest} elements $c$ in this Dirichlet series representation of $\phi_X(s)$. 

Write $c_j := \min \cC_{j j}$ and consider the number $c_1 c_2 \dots c_k$, $c_j = \min \cC_{j j}$. Then, by the definition of the determinant, there will be one $n \in \N$ such that $b_n = c_1 c_2 \dots c_k$ in the sum \eqref{eq:dirichlet}. Now, this number $c_j$ comes from isometric circles associated to group elements $\gamma \in \Gamma$, and in fact $c_j^{-1}$ equals to the radius of the largest isometric circle over $\gamma \in \Gamma$, see e.g. \cite[Section 2.6]{Iwa02}. In particular, by \cite[(2.31)]{Iwa02}, $c_j$ satisfies the bound: $c_j \leq \Vol(X)$. Therefore, we have the following estimate for $b_1$ in terms of volume of $X$:
\begin{align}\label{sc1}b_1 \leq \Vol(X)^k\end{align}

Having described $b_1 = b_1(X)$ from \eqref{eq:scatteringformula}, we can now move to adapt it to prove our claim. In the sum over the poles in \eqref{eq:scatteringformula}, let $S_1(r)$ be the sum over finite number of the poles in $(1/2,1]$ and $S_2(r)$ be the sum over the poles with $\mathrm{Re} s < 1/2$. Note all the poles are either in $\{\mathrm{Re} s < 1/2\}$ or in $(1/2,1]$, and there are only finitely many in the latter case. Then in particular $S_1(r)\geq 0$, $S_2(r) \leq 0$ and by \eqref{eq:scatteringformula} we have
\begin{align}\label{scatteringidentity}\frac{-\phi_X'(\frac12 + ir)}{\phi_X(\frac12+ ir)} = 2\log b_1 + S_1(r) + S_2(r).\end{align}
Then if $f \geq 0$, using $S_1(r)\geq 0$ and $S_2(r) \leq 0$, we obtain:
\begin{align}\label{sc2}\int f(r) \Big|\frac{\phi_X'(\frac12 + ir)}{\phi_X(\frac12+ ir)} \Big| \, dr \leq 2\|f\|_1|\log b_1| + \int f(r) S_1(r) dr - \int f(r) S_2(r) \, dr.\end{align}
Using again \eqref{scatteringidentity} and $-f(r) S_2(r) \geq 0$ and $S_1(r) \geq 0$, we also have the following estimate:
\begin{align}\label{sc3}-\int f(r) S_2(r)  \, dr \leq \Big| \int f(r) \frac{-\phi_X'(\frac12 + ir)}{\phi_X(\frac12+ ir)} dr\Big|  + 2\|f\|_1|\log b_1| + \int f(r) S_1(r)\, dr\end{align}
Moreover, by a result of Otal and Rosas \cite[Theorem 2]{OtalRosas}, we know that the number of eigenvalues $\leq 1/4$ is at most $2g - 2 +k \lesssim \Vol(X)$ including possible multiplicity for each eigenvalue. On the other hand, for each pole $\rho_j \in [1/2,1]$, we know that $0 \leq \rho_j (1-\rho_j) \leq 1/4$ is an eigenvalue of the Laplacian and that the multiplicity of $\rho_j$ is at most the multiplicity of $\rho_j (1-\rho_j)$ as an eigenvalue of the Laplacian, see e.g. \cite[Theorem 4.1 combined with Theorem 3.1]{PhilipsSarnak} or \cite[(3.33) on page 299]{Hejhal}. Hence the total number of poles including multiplicities in $[1/2,1]$ is bounded by $ \lesssim \Vol(X)$ giving us
$$0 \leq S_1(r) \lesssim \Vol(X)$$
so
\begin{align}\label{sc4}\int f(r) S_1(r) \, dr \lesssim \Vol(X) \|f\|_1.\end{align}
Combining \eqref{sc1}, \eqref{sc2}, \eqref{sc3} and \eqref{sc4} gives us 
$$\int f(r) \Big|\frac{\phi_X'(\frac12 + ir)}{\phi_X(\frac12+ ir)} \Big| \, dr  - \Big| \int f(r) \frac{-\phi_X'(\frac12 + ir)}{\phi_X(\frac12+ ir)} dr\Big| \lesssim \Big( k|\log \Vol(X)| + \Vol(X)\Big) \|f\|_1.$$
\end{proof}

\subsection{Proof of Theorem \ref{thm:mainBS}}
We now deduce Theorem \ref{thm:mainBS} from Theorem \ref{thm:main}. Let us first discuss on the term $\rho(\lambda_1(X_n))$ from the statement. 

\begin{remark}\label{rmk:unifsp}
Suppose $\cB$ is a set of hyperbolic surfaces $X$ of finite area such that $\lambda_1(X) \geq \eps_0$ for some fixed constant $\eps_0 > 0$. Thus the surfaces in $\cB$ have uniform spectral gap with uniform lower bound given by $\eps_0$. Now, as discussed in \cite{GN15,Nev98}, uniform spectral gap for all surfaces $X \in \cB$ implies 
$$c_0 := \sup \{q(\pi_X|_{L^2_0(X)}) : X \in \cB\} < \infty,$$
where $q(\pi_X|_{L^2_0(X)})$ is the integrability exponent of the regular representation $\pi_X|_{L^2_0(X)}$ on $L^2_0(X) = \{f \in L^2(X) : \int f \, d\mu = 0\}$ and $\pi_X f(x) := \int f(g^{-1} x) \, dg$, $x \in X$, $f \in L^2(X)$ and $dg$ is the Haar measure on $\mathrm{PSL}(2,\R)$. The quantity $c_0$ only depends on the set $\cB$. On the other hand, as we saw in the proof of Theorem \ref{thm:main}, the quantity $\rho(\lambda_1(X))$ equals to the quantity $\theta$ from Theorem \ref{thm:nevo}. In \cite[Section 4]{GN15} and in particular \cite{Nev98} where Theorem \ref{thm:nevo} is proved, we see that $\theta$ can be chosen to be any of the numbers $n^{-1}(1-r^{-1})$, where $1 \leq r < 2$ and $n \in \N$ is the smallest even integer such that $n \geq 2c_0$. Setting for example $r = 4/3$, we see that $\rho(\lambda_1(X)) = \theta \geq 1/(4c_0 + 8)$, which is an $X$ independent lower bound. 
\end{remark}

Thus by Remark \ref{rmk:unifsp}, as $\lambda_1(X_n)$ is uniformly bounded from below, we know that also $\rho(\lambda_1(X_n))$ is. By assumption the number of cusps $k_n = k(X_n)$ of $X_n$ satisfies for some $0 \leq \kappa < 1/2$, $k_n = O(g_n^\kappa)$ when $n\to +\infty$, and in particular $\Vol(X_n) = O(g_n)$. Putting this together with the uniform bounds on the systole giving 
$$\inj(X(Y e^{R_n/2})) = \frac{1}{2}\min\{\sys(X),e^{-R_n/2} Y^{-1}\} \gtrsim e^{-R_n/2} Y^{-1},$$ 
the spectral gap and the test functions $a_n$ we obtain

\begin{align*}
 \frac{N(X_n,I) + M(X_n,I)}{\Vol(X_n)} \Var_{X_n,I}(a_n) 
& \lesssim_I  \frac{\max\{N(X_n,I), g_n^\kappa\}^{\frac12}}{\Vol(X_n)^{\frac12}}\left( \frac{1}{R_n} + Y e^{5R_n/2} \frac{\Vol((X_n)_{\leq R_n})}{\Vol(X_n)}  \right)^{\frac12} \\
&\quad + g_n^{2\kappa - 1} + g_n^{\kappa - 2} M(X_n,I) + g_n^{2\kappa-2} \log g_n,
\end{align*}
where we have chosen a sequence $R_n \to +\infty$ of Benjamini-Schramm convergence parameters such that  $e^{5R_n/2} \frac{\Vol((X)_{\leq R_n})}{\Vol(X_n)} \to 0$ when $n \to +\infty$.

By the spectral asymptotic estimate Theorem \ref{thm:weyl} proved in Section \ref{s:spectralconv} we know that
$$ \lim_{n \to +\infty} \frac{N(X_n,I) + M(X_n,I)}{\Vol(X_n)} = O(1) $$
which together with the previous bound on $\Var_{X_n,I}(a_n)$ gives 
$$ \Var_{X_n,I}(a_n) \to 0$$
when $n \to +\infty$.

\section{Proof of the spectral convergence}\label{s:spectralconv}

We show in this section the following level aspect analogue of the Weyl law:

\begin{thm}\label{t:spectralconv}
Let $X_n$ be a sequence of finite area hyperbolic surfaces of genus $g_n$ and number of cusps $k_n = o(g_n)$, Benjamini-Schramm converging to the plane $\H$, and such that the length of the shortest closed geodesic (the systole) is uniformly bounded from below by a constant. Then for any compact interval $I \subset [0,\infty)$ we have
$$N(X_n,I) + M(X_n,I) \sim \Vol(X_n),$$
when $n \to +\infty$.
\end{thm}

In Section \ref{sec:quantspectralconv}, we will prove a quantitative and somewhat stronger version of this result. The proof of Theorem \ref{t:spectralconv} uses the following proposition about the asymptotics of the heat trace when $n \to +\infty$.

\begin{prop}\label{p:heat}
Let $X_n$ be a sequence of finite area hyperbolic surfaces of genus $g_n$ and number of cusps $k_n = o(g_n)$, Benjamini-Schramm converging to the plane $\H$, and such that the length of the shortest closed geodesic (the systole) is uniformly bounded from below by a constant.  Fix $t > 0$. Then
\begin{align*}
& \quad \,\, \lim_{n \to \infty}\frac1{\Vol(X_n)} \left( \sum_{j = 0}^\infty e^{-t\lambda_j^{(n)}} + \frac1{4\pi} \int_{-\infty}^{+\infty} \frac{-\varphi_n'}{\varphi_n} \left(\frac12 + ir \right) e^{-t(\frac14 + r^2)} \, dr \right) \\
 & = \frac{1}{4\pi} \int_{-\infty}^{+\infty} e^{-t(1/4 +r^2)} \tanh(\pi r) r \, dr ,
\end{align*}
where $\varphi_n := \phi_{X_n}$ is the determinant of the scattering matrix associated with $X_n$.
\end{prop}

\begin{proof}[Proof of Proposition \ref{p:heat}]
The proof is based on Selberg trace formula for finite area hyperbolic surfaces (See \cite[Chapter 10]{Iwa02}). One of the main difficulties is to deal with conjugacy classes of parabolic elements, corresponding to cusps. The idea is to use a cut-off at a height $Y$ in the cusps and to compute the truncated trace spectrally and geometrically. Diverging terms in $Y$ then cancel each other and what remains is the final trace formula. The diverging terms only come from the parabolic classes and so we will use the final form of the trace formula \cite[Theorem 10.2]{Iwa02} for every term apart from the ones corresponding to hyperbolic elements in $\Gamma - \{ \id \}$ (denoted by $\mathcal H_n$), whose treatment does not require any cut-off. For the hyperbolic terms instead of using the final form as a sum over closed geodesics, we revert to the integral of a kernel, to which we can apply BS-convergence. 

We have the formula
\begin{align*} 
&\sum_j h_t(r_j) + \frac1{4\pi} \int_{-\infty}^{\infty} h_t(r) \frac{-\varphi_n'}{\varphi_n}\Big(\frac12 + ir\Big) \, dr\\
&\quad = \frac{|F_n|}{4\pi} \int_{-\infty}^{\infty} h_t(r) r \, \tanh(\pi r) \, dr + \sum_{\gamma \in \mathcal H_n} \int_{F_n} \mathbf{k}_t(d(z,\gamma z)) \, d\mu(z) \\
&\qquad + \frac{h_t(0)}4 \text{Tr} \left( I - \Phi_n\big(\frac12\big) \right) 
- |\mathfrak{C}_n| g_t(0) \log 2 - \frac{|\mathfrak{C}_n|}{2\pi} \int_{-\infty}^{\infty} h_t(r) \, \psi(1+ ir) \, dr. 
\end{align*}

Here $\mathbf{k}_t$ is the heat kernel, $h_t(r) = e^{-t(\frac14 + r^2)}$ its Selberg transform and $g_t = \hat h_t$ the Fourier transform,  $F_n$ is a fundamental domain and $|\mathfrak{C}_n|$ is the number of inequivalent cusps, $\psi(s) = \Gamma'(s)/\Gamma(s)$, and $\Phi_n(s) := \Phi_{X_n}(s)$ is the scattering matrix of $X_n$ (See \cite{Iwa02} for background). For $\text{Re}(s) = \frac12$ the scattering matrix $\Phi_n(s)$ is unitary (see \cite[Theorem 6.6]{Iwa02}) and its rank is equal to the number of cusps, so the term $\text{Tr} \left( I - \Phi_n\big(\frac12\big) \right)$ is controlled by the number of cusps $|\mathfrak{C}_n|$. By assumption on $X_n$ we have $ \frac{|\mathfrak{C}_n|}{|F_n|} \to 0$ when $n \to +\infty$.



The treatment of the hyperbolic terms follows exactly the proof of the compact case: Proposition 9.5 in \cite{LS17}. Using the heat kernel estimate $\mathbf{k}_t(\rho) \lesssim_t e^{-\rho/(8t)}$ we can show that for any $R > 0$

$$ \frac1{|F_n|} \sum_{\gamma \in \mathcal H_n} \int_{F_n} k_n(d(z,\gamma z)) \, d\mu(z) = O\left(\frac{e^{-R^2}}{\sys(X_n)}\right)
+ O \left( \frac1{\sys(X_n)} \frac{\Vol( (X_n)_{\leq R} )}{\Vol(X_n)} \right)$$
where $\sys(X_n) = \inf_{z\in X_n} \{ d(z,\gamma z), \gamma \in \mathcal H_n \}$ is the length of the shortest closed geodesic (systole). By Benjamini-Schramm convergence, we can take a sequence $R_n \to +\infty$ such that
$$\frac{\Vol( (X_n)_{< R_n} )}{\Vol(X_n)}  \to 0$$
when $n \to +\infty$. This concludes the proof of Proposition \ref{p:heat}.
\end{proof}

From Proposition \ref{p:heat} we can deduce Theorem \ref{t:spectralconv} by an approximation argument identical to the proof of Theorem 9.2 in \cite{LS17}. Indeed we can approximate any function $f$ supported on the union of compact intervals $\tau^{-1}(I)$ for a compact interval $I$ by linear combinations of exponential functions $x \mapsto e^{-tx}$ with $t >0$ using the Stone-Weierstrass theorem.

\section{Quantitative spectral convergence} \label{sec:quantspectralconv}

We now turn towards random surfaces and Theorem \ref{thm:largegenus}. Before we prove it we adapt in this section the results of Monk \cite{Monk} to finite area surfaces. In \cite{Monk}, a quantitative version of the spectral convergence is proved for compact hyperbolic surfaces. We will need such quantitative convergence because we want uniformity over the probability sets we consider, and the previous section does not give us that. The argument of \cite{Monk} extends to non-compact surfaces because the terms in the trace formula arising from the parabolic elements are well-behaved under the Benjamini-Schramm convergence assumption. We reproduce here the steps of the argument of \cite{Monk}, emphasising the main differences.

Let $k(g) = O(g^{\kappa})$ for some $0 \leq \kappa < 1/2$ and $\cM_{g,k(g)}$ be the moduli space of hyperbolic surfaces of genus $g$ with $k(g)$ cusps, recall Section \ref{sec:introrandom} for the definition. Define the subset $\cA_{g,k(g)} \subset \cM_{g,k(g)}$ of surfaces $X$ such that
\begin{enumerate}
\item $$\frac{\Vol\left( (X)_{\leq \frac16 \log g} \right)}{\Vol(X)} \leq g^{-\frac13}$$\label{itm:BS}
\item $$\mathrm{sys}(X) \geq g^{-\frac1{24}}(\log g)^{\frac12}.$$\label{itm:sys}
\end{enumerate}
We first remark that these assumptions are satisfied with high probability when $g$ is large.

\begin{thm}\label{thm:sp1}
Assume $k(g) = O(g^{\kappa})$ for some $0 \leq \kappa < 1/2$, then $\cA_{g,k(g)}$
satisfies that $\P_{g,k(g)}(\cA_{g,k(g)}) = 1 - O\left(g^{-\beta}\right)$
for some $\beta > 0$.
\end{thm}

The probability of the event of surfaces $X$ satisfying \eqref{itm:BS} was proved in \cite[Corollary 4.4]{MonkThesis}. The systole part \eqref{itm:sys} was mentioned by Mirzakhani \cite[Theorem 4.2]{Mi} but there was no \textit{quantitative} dependence on the the number of cusps $k$ and the proof was given only for compact hyperbolic surfaces. Since we have growing number of cusps, we need a quantitative version. We provide this in Appendix \ref{a:systoleapp} (Lemma \ref{lma:noncompactsystole}). Together with \cite[Corollary 4.4]{MonkThesis}, this gives a proof of Theorem \ref{thm:sp1}.

We now prove a quantitative spectral convergence theorem for surfaces in $\cA_{g,k(g)}$, extending the one proved for compact surfaces in \cite{Monk}. 

\begin{thm} \label{thm:sp2}
Let $I = [a,b] \subset (0, +\infty)$. If $X \in \cA_{g,k(g)}$ with $k(g) = O(g^{\kappa})$ for some $0 \leq \kappa < 1/2$, then we have
$$\frac{N(X,I) + M(X,I)}{|X|} = \frac{1}{4\pi}\int_I \tanh\Big(\pi \sqrt{\lambda-\tfrac{1}{4}}\Big) \, d\lambda + R(X,I),$$
where 
$$-C \sqrt{\frac{b+1}{\log g}} \leq R(X,I) \leq C \sqrt{\frac{b+1}{\log g}} \log \left( 2 + (b-a) \sqrt{\frac{\log g}{b+1}} \right)^{1/2}$$
for some implicit constant $C > 0$ that only depends on $\kappa > 0$.
\end{thm}

The proof is based on applying Selberg's trace formula to well chosen test functions. We use the same test functions as \cite{Monk}. One of them will give the result for $\frac12 \leq a < b$ (the test function $H_t$ below) and the other for $b \leq 1$ (the test function $\tilde H_t$ below). Let us now state Selberg's trace formula that plays a crucial role in the proof of Theorem \ref{thm:sp2}. It can be extracted from \cite[Chapter 10]{Iwa02} as is explained at the beginning of the proof of Proposition \ref{p:heat}.

\begin{lemma}[Selberg's trace formula for hyperbolic surfaces of finite area]\label{lma:SelbergTrace}
We will call a function $h: \C \to \C$ admissible if it satisfies the following properties:
\begin{enumerate}
\item $h(-r) = h(r)$ for any $r \in \C$;
\item $h$ is holomorphic in the strip $|\Im z| \leq \frac12 + \epsilon$ for some $\epsilon > 0$;
\item for any $r$ in the strip $h(r) \lesssim (1 + |r|^2)^{-1-\epsilon}$.
\end{enumerate}
For any admissible function $h : \C \to \C$ we have:
\begin{align*} 
&\sum_j h(r_j) + \frac1{4\pi} \int_{-\infty}^{\infty} h(r) \frac{-\varphi'_X}{\varphi_X}\Big(\frac12 + ir\Big) \, dr\\
&\quad = \frac{|F|}{4\pi} \int_{-\infty}^{\infty} h(r) r \, \tanh(\pi r) \, dr + \sum_{\gamma \in \mathcal H} \int_{F} \mathbf{k}(d(z,\gamma z)) \, d\mu(z) \\
&\qquad + \frac{h(0)}4 \mathrm{Tr} \left( I - \Phi_X\big(\frac12\big) \right) 
- |\mathfrak{C}| g(0) \log 2 - \frac{|\mathfrak{C}|}{2\pi} \int_{-\infty}^{\infty} h(r) \, \frac{\Gamma'}{\Gamma}(1+ ir) \, dr,
\end{align*}
where $\mathbf{k}$ is the inverse Selberg transform of $h$ and $g$ is the inverse Fourier transform $g(u) = \frac1{2\pi}  \int_{-\infty}^{+\infty} e^{isu} h(s) \, ds$ and $\Gamma$ is the Gamma-function. The set $F$ is a fundamental domain of the surface $X$ and $|\mathfrak{C}|$ is the number of inequivalent cusps in $X$.
\end{lemma}

We will apply this trace formula for two different test functions $H_t$ and $\tilde H_t$ below, depending on which part of Theorem \ref{thm:sp2} we are proving. 

Let us first consider the test function $H_t = H_t^{a,b}$, which we define for all $\frac{1}{4} < a < b$, but due to Lemma \ref{lma:12need} later, we can use it effectively only for $a \geq 1/2$. Fix $\frac{1}{4} < a < b$ and define $0 < \alpha < \beta$ such that 
$$a = \frac14 + \alpha^2 \quad \text{and}\quad b = \frac14 + \beta^2.$$ 
Fix $t > 0$, which, we will eventually set $t := \frac{\sqrt{\log g}}{4 \sqrt{3}}$. Define, as in \cite[Section 4.1]{Monk}, the function
$$ h_t(r) := \mathbf{1}_{[\alpha, \beta]} \ast v_t (r) = \frac{t}{\sqrt{\pi}} \int_\alpha^\beta \exp \left(-t^2(r - \rho)^2\right)\, d\rho 
= \frac1{\sqrt{\pi}} \int_{t(\alpha - r)}^{t(\beta - r)} \exp (-\rho^2) \, d\rho,$$
where $v_t(x) = \frac{t}{\sqrt{\pi}} e^{-t^2 x^2}$, which then gives us the holomorphic and even test function:
$$ H_t(r) := h_t(r) + h_t(-r).$$
By Lemma \cite[Lemma 19]{Monk}, this test function has the proper decay in order to be admissible for the Selberg trace formula.

Before we apply the Selberg trace formula, we will give the following lemma that allows us to compare $H_t$ quantitatively to the modified indicator function 
$$\tilde \1_{[\alpha,\beta]}(r) =  \begin{cases} 1, & \alpha < r < \beta \\
\frac{1}{2}, & r \in \{\alpha,\beta\} \\
0, & r < \alpha \text{ or } r > \beta
\end{cases}$$
that we end up using several times in the proof.
\begin{lemma}\label{lma:decayHt}
When $0 < \alpha < \beta$, $t > 0$ and $r \in \R$, we have:
$$|H_t(r) - \tilde \1_{[\alpha,\beta]}(r)| \leq \begin{cases} s(t|r+\alpha|) + s(t|r-\alpha|), & r < \alpha \text{ or } r = \beta, \\
s(t|r+\alpha|)  + s(t|r-\alpha|) +  s(t|r-\beta|), & \alpha < r < \beta, \\
s(t|r+\alpha|)  + s(t|r-\beta|), & r > \beta \text{ or } r = \alpha,
\end{cases}$$
where $s(\rho) := \frac{1}{2\sqrt{\pi} \rho} e^{-\rho^2}$.
\end{lemma}

\begin{proof}
The proof is a combination of triangle inequality
$$|H_t(r) - \tilde \1_{[\alpha,\beta]}(r)| \leq |h_t(-r) - \tilde \1_{[\alpha,\beta]}(-r)| + |h_t(r) - \tilde \1_{[\alpha,\beta]}(r)|.$$
and \cite[Lemma 21]{Monk} that says
$$|h_t(r) - \tilde \1_{[\alpha,\beta]}(r)| \leq \begin{cases} s(t|r-\alpha|), & r < \alpha \text{ or } r = \beta, \\
 s(t|r-\alpha|) +  s(t|r-\beta|), & \alpha < r < \beta, \\
 s(t|r-\beta|), & r > \beta \text{ or } r = \alpha,
\end{cases}$$
where $s(\rho) =  \frac{1}{2\sqrt{\pi} \rho} e^{-\rho^2}$.
\end{proof}

Let us now apply the Selberg trace formula to $H_t$, which gives us:
\begin{align*} 
&\frac1{|F|} \left( \sum_j H_t(r_j) + \frac1{4\pi} \int_{-\infty}^{\infty} H_t(r) \frac{-\varphi'_X}{\varphi_X}\Big(\frac12 + ir\Big) \, dr \right)\\
&\quad = \frac{1}{2\pi} \int_{\alpha}^{\beta} r \, \tanh(\pi r) \, dr + \cR(t,a,b) + \cR_K(X,t,a,b) + \cR_{NC}(X,t,a,b)
\end{align*}
where
$$ \cR(t,a,b) :=  \frac{1}{2\pi} \int_0^{+\infty} (H_t(r) - \mathbf{1}_{[\alpha, \beta]}(r)) \, r \, \tanh(\pi r) \, dr $$
$$ \cR_K(X,t,a,b)  := \frac1{|F|} \sum_{\gamma \in \mathcal H} \int_{F} K_t(d(z,\gamma z)) \, d\mu(z)$$
where $K_t$ is the inverse Selberg transform of $H_t$, and
$$  \cR_{NC}(X,t,a,b) :=\frac{H_t(0)}{4|F|} \text{Tr} \left( I - \Phi_X\big(\frac12\big) \right) 
- \frac{|\mathfrak{C}|}{|F|} G_t(0) \log 2 - \frac{|\mathfrak{C}|}{2\pi |F|} \int_{-\infty}^{\infty} H_t(r) \, \frac{\Gamma'}{\Gamma}(1+ ir) \, dr,  $$
with $G_t$ the inverse Fourier transform of $H_t$. We now proceed to estimate these three quantities.

The first one $\cR(t,a,b)$ has no dependence on the surface, and it can be estimated in exactly the same way as is done in \cite{Monk}. We have
\begin{lemma}[{\cite[Proposition 20]{Monk}}] For any $t \geq \frac1{10}$, and any $\frac14 < a < b$
 $$\cR(t,a,b) = O\left( \frac{\sqrt{b}}{t} \right).$$
\end{lemma}

For $\cR_K$ the estimate is virtually the same as in \cite{Monk}, except that the injectivity radius is replaced by the systole. To see that, let us recall that for a general hyperbolic surface $X = \Gamma \backslash \H$ the injectivity radius can be written as 
$$\inj_X = \frac12 \inf_{z\in X} \inf_{\gamma \in \Gamma} d(z,\gamma z)$$
and the systole as
$$\sys(X) = \inf_{z\in X} \inf_{\gamma \in \cH} d(z,\gamma z),$$
where $\cH$ is the set of hyperbolic elements in $\Gamma$. On compact surfaces we have $\Gamma = \cH$ so $\sys(X) = 2\inj_X$. On a non-compact surface, the relevant quantity to estimate $\cR_K$ is the systole, as $\cR_K$ involves the sum over hyperbolic elements. We use the following lemma.

\begin{lemma}\label{l:countingsys}
For a hyperbolic surface $X = \Gamma \backslash \H$, if $\mathcal H$ denotes the set of hyperbolic elements of $\Gamma$, we have
$$\# \{ \gamma \in \mathcal H :  d(z, \gamma z) \leq j \} \lesssim \frac{e^j}{\min\{1,\sys(X)^2\}}$$
for any $z \in \H$ and $j > 0$.
\end{lemma}

The proof of Lemma \ref{l:countingsys} consists of exactly the same counting argument as in the proof of Lemma \ref{lma:normsurface}, but instead of counting over all elements of $\Gamma$ and considering elements with $z \in X(Y)$, where, recall $ X(Y) = X \setminus \bigcup_{\mathfrak{b}} X_{\mathfrak{b}} (Y),$ and $X_{\mathfrak{b}} (Y)$ is the cuspidal zone associated with $\mathfrak{b}$, we count only over hyperbolic elements and take $z \in \H$, the injectivity radius of the thick part $\inj_{X(Y)}$ gets therefore replaced by the systole $\sys(X)$.

Once this is understood, the extension of \cite[Lemma 24 and Proposition 25]{Monk} is immediate:
\begin{lemma}[{\cite[Proposition 25]{Monk}}]
There exists $g_0 \in \N$ such that for all $g \geq_0$,  $X \in \cA_{g,k(g)}$ and pair $\frac14 < a < b$ we have
$$\cR_K(X,t,a,b) \lesssim \sqrt{ \frac{b}{\log g}}$$
whenever $t := \frac{\sqrt{\log g}}{4 \sqrt{3}}$.
\end{lemma}

\begin{proof}
The computation is the same as in the compact case of \cite{Monk}, we sketch the argument without going into the computational details. We need to use the estimate on the kernel $K_t(\rho)$. From \cite[Lemma 23]{Monk} we have that for any $r \in (0,3)$, $t \geq \frac1{10}$ and $\rho \geq r$,
$$ K_t(\rho) \lesssim \frac{t \sqrt{b}}{r^2} \exp\left( -\frac{\rho^2}{4t^2} \right).$$
Take $r \leq \sys(X)$ and $L \geq 8t^2$. We estimate $\cR_K(X,t,a,b)$ splitting between the points $(F)_{<L}$ with radius of injectivity less than $L$ and the points $(F)_{\geq L}$ with radius of injectivity greater than $L$. For $(F)_{\geq L}$ we have, decomposing into a series and using Lemma \ref{l:countingsys}
\begin{align*}
&\frac1{|F|} \int_{(F)_{\geq L}} \sum_{j \geq L} \sum_{ j \leq d(z,\gamma z) < j+1} K_t(d(z,\gamma z))\, d\mu(z) \\
&\quad \lesssim \frac1{|F|} \int_{(F)_{\geq L}} \sum_{j \geq L} \frac{e^j}{r^2} \frac{t \sqrt{b}}{r^2} e^{-\frac{j^2}{4 t^2}} \, d\mu(z),
\end{align*}
and, recall, in the sum the elements $\gamma$ are hyperbolic. Apart from taking $r \leq \sys(X)$ instead of $r \leq \inj_X$, there is no difference with the compact case and given $L \geq 8t^2$ we can bound this quantity by
$$ \frac{t \sqrt{b}}{r^4} e^{-L}$$
Similarly for $(F)_{< L}$ we have a bound in 
$$ \frac{t^3 \sqrt{b}}{r^4} \frac{|(F)_{< L}|}{|F|} e^{L}.$$
We set $t = \frac{\sqrt{\log g}}{4 \sqrt{3}}$, and $L = \frac16 \log g = 8 t^2$. We can then use that in $\cA_{g,k(g)}$ we have $$\frac{|(F)_{< \frac16 \log g}|}{|F|} \leq g^{-\frac13}$$
and
$$\mathrm{sys}(X) \geq g^{-\frac1{24}}(\log g)^{\frac12},$$
which gives the required bound.
\end{proof}

The term $\cR_{NC}$ arises from the non-compactness (the parabolic elements of $\Gamma$) and therefore does not appear in \cite{Monk}. First note that for $\text{Re}(s) = \frac12$ the scattering matrix $\Phi_X(s)$ is unitary (see \cite[Theorem 6.6]{Iwa02}) and its rank is equal to the number of cusps, so the term $\text{Tr} \left( I - \Phi_X\big(\frac12\big) \right)$ is controlled by the number of cusps $|\mathfrak{C}|$. We thus have
$$  |\cR_{NC}(X,t,a,b)| \lesssim \frac{|\mathfrak{C}|}{|F|} \left( |H_t(0)| + |G_t(0)| + \left| \int_{-\infty}^{\infty} H_t(r) \, \frac{\Gamma'}{\Gamma}(1+ ir) \, dr \right| \right)  $$

%
%

We obtain the following bound.
\begin{lemma} \label{lma:NCbound}There exits $g_0 \in \N$ such that for all $g \geq g_0$ and $X \in  \cA_{g,k(g)}$ and every pair $\frac{1}{4} < a < b$ and $\epsilon >0$, we have
$$\cR_{NC}(X,t,a,b) \lesssim_\eps \frac{c_{a,b}}{\sqrt{g}}$$
for  $t := \frac{\sqrt{\log g}}{4 \sqrt{3}}$ and $c_{a,b} := \log\Big( \frac{\pi\sqrt{b - \frac{1}{4}}}{\sinh(\pi(\sqrt{a - \frac{1}{4}}))}\Big) + 1 + \sqrt{b}$.
\end{lemma}
\begin{proof}
If $X \in \cA_{g,k(g)}$, we know that for some $0 \leq \kappa < 1/2$ the number of cusps $|\mathfrak{C}| = k(g) = O(g^\kappa)$, so in particular, $ \frac{|\mathfrak{C}|}{|F|} \lesssim g^{-1+\kappa}$. Thus the claim follows if we can prove the estimate
$$|H_t(0)| + |G_t(0)| + \left| \int_{-\infty}^{\infty} H_t(r) \, \frac{\Gamma'}{\Gamma}(1+ ir) \, dr \right| \lesssim_\eps  c_{a,b} \sqrt{\log g}.$$
For the term $H_t(0)$ we have the estimate
$$|H_t(0)| = 2 h_t(0) \leq \frac2{\pi} \int_{-\infty}^{+\infty} e^{-\rho^2} d\rho \leq 2.$$
For $G_t(0)$ we can compute that $g_t$ is given by
$$ g_t(u) = \frac1{\pi} \left( \frac{\sin(\beta u)}{u} - \frac{\sin(\alpha u)}{u} \right) e^{-\frac{u^2}{4t^2}},$$
and therefore
$$|G_t(0)| = 2 |g_t(0)| = \frac{2 (\beta - \alpha)}{\pi} = O\left(\sqrt{b} \right).$$
Finally, let us look at the integral $\int_{-\infty}^{\infty} H_t(r) \, \frac{\Gamma'}{\Gamma}(1+ ir) \, dr$. First we can split
\begin{align}\label{eq:splitGamma}\int_{-\infty}^{\infty} H_t(r) \, \frac{\Gamma'}{\Gamma}(1+ ir) \, dr = \int_{\alpha}^{\beta} \frac{\Gamma'}{\Gamma}(1+ ir) \, dr + \int_{-\infty}^{\infty} (H_t(r) - \mathbf{1}_{[\alpha, \beta]}(r)) \, \frac{\Gamma'}{\Gamma}(1+ ir) \, dr.\end{align}
Let us estimate the first integral $\int_{\alpha}^{\beta} \frac{\Gamma'}{\Gamma}(1+ ir) \, dr$ in \eqref{eq:splitGamma}. Since the digamma function $s \mapsto \Gamma'(s)/\Gamma(s)$ is holomorphic in a neighbourhood of the curve $L = \{1+ir : r \in [\alpha,\beta]\}$ in $\C$, the fundamental theorem of calculus for complex line integrals applied to the curve $L$ gives us 
\begin{align*}
\int_{\alpha}^{\beta} \frac{\Gamma'}{\Gamma}(1+ ir) \, dr =  i\left( \mathrm{Log}\,\Gamma(1+ i\beta) - \mathrm{Log}\,\Gamma(1+ i\alpha) \right)
\end{align*}
where $\mathrm{Log}(w) = \log |w| + i\mathrm{Arg}(w)$, $\mathrm{Arg}(w) \in (0,2\pi]$, is the principal branch of the complex natural logarithm. By \cite[6.1.31]{Ab}, we have $|\Gamma(1+iy)|^2 = \frac{\pi y}{\sinh (\pi y)}$ so
\begin{align*}
\left| \mathrm{Log}\,\Gamma(1+ i\beta) -  \mathrm{Log}\,\Gamma(1+ i\alpha) \right| &\leq |\mathrm{Log}\, \Gamma(1+i\beta)| + |\mathrm{Log}\, \Gamma(1+i\alpha)|\\
&\leq \log|\Gamma(1+i\beta)| + \log|\Gamma(1+i\alpha)| + 2\pi\\
& =  \frac{1}{2}\log\frac{\pi \beta}{\sinh (\pi \beta)} + \frac{1}{2}\log \frac{\pi \alpha}{\sinh (\pi \alpha)} + 2\pi.
\end{align*}
Since $\sinh(\pi \beta) \geq \sinh(\pi \alpha)$, $\alpha = \sqrt{a - \frac{1}{4}}$, $\beta = \sqrt{b - \frac{1}{4}}$, we have
$$\left| \int_{\alpha}^{\beta} \frac{\Gamma'}{\Gamma}(1+ ir) \, dr \right| \lesssim \log\Big( \frac{\pi\sqrt{b - \frac{1}{4}}}{\sinh(\pi(\sqrt{a - \frac{1}{4}}))}\Big) + 1.$$

We are left with estimating the second integral $ \int_{-\infty}^{\infty} (H_t(r) - \mathbf{1}_{[\alpha, \beta]}(r)) \, \frac{\Gamma'}{\Gamma}(1+ ir) \, dr$ from \eqref{eq:splitGamma}. If we apply Lemma \ref{lma:decayHt} to bound $|H_t(r) - \mathbf{1}_{[\alpha, \beta]}(r)|$, we end up having singularities near $\pm\alpha$ and $\pm\beta$ so we have to truncate the integration using a parameter $0 < \eps < (\beta-\alpha)/2$, which we set $\eps := 1/t$, where, recall $t = \frac{\sqrt{\log g}}{4 \sqrt{3}}$, which will force us to assume the genus $g$ is large enough. Define
$$C_{\eps} := \R \setminus \Big([\alpha - \eps,\alpha+\eps] \cup [\beta-\eps,\beta+\eps] \cup [-\alpha - \eps,-\alpha+\eps] \cup [-\beta-\eps,-\beta+\eps]  \Big).$$
Since $h_t(r) = \frac1{\sqrt{\pi}} \int_{t(\alpha - r)}^{t(\beta - r)} \exp (-\rho^2) \, d\rho$ and $H_t(r) = h_t(r) + h_t(-r)$, we can use a trivial estimate $|H_t(r) - \mathbf{1}_{[\alpha, \beta]}(r)| \leq 2(\beta-\alpha) t + 2$, $|\frac{\Gamma'}{\Gamma}(1+ ir)| \lesssim \log |r|$ for all $r \in \R$ by \cite[6.3.18]{Ab}, and that $t = \frac{\sqrt{\log g}}{4 \sqrt{3}}$, which together imply
$$ \int_{\R \setminus C_\eps} (H_t(r) - \mathbf{1}_{[\alpha, \beta]}(r)) \, \frac{\Gamma'}{\Gamma}(1+ ir) \, dr \lesssim_{\alpha,\beta} \eps \sqrt{\log g}.$$

Let us now complete the proof by dealing with the integral over $C_\eps$, where we will heavily use the estimate from Lemma \ref{lma:decayHt} that says any $r \in \R$ has the bound
$$|H_t(r) - \tilde \1_{[\alpha,\beta]}(r)| \leq \begin{cases} s(t|r+\alpha|) + s(t|r-\alpha|), & r < \alpha \text{ or } r = \beta, \\
s(t|r+\alpha|)  + s(t|r-\alpha|) +  s(t|r-\beta|), & \alpha < r < \beta, \\
s(t|r+\alpha|)  + s(t|r-\beta|), & r > \beta \text{ or } r = \alpha,
\end{cases}$$
where $s(\rho) = \frac{1}{2\sqrt{\pi} \rho} e^{-\rho^2}$ is a decreasing function in $\rho$.

Let us first consider the integration over $(\beta + \eps,\infty)$:
$$\int_{\beta + \eps}^\infty |H_t(r) - \tilde \1_{[\alpha,\beta]}(r)| \Big|\frac{\Gamma'}{\Gamma}(1+ ir)\Big| \, dr.$$
Since $|\frac{\Gamma'}{\Gamma}(1+ ir)| \lesssim \log |r|$ for all $r \in \R$ by \cite[6.3.18]{Ab}, a change of variable $\rho = t(r-\beta)$ gives us that we can bound this by a constant multiple of
$$\int_{\beta + \eps }^\infty s(t|r-\beta|) \log |r| \, dr \lesssim \frac{1}{\eps t^2} \int_{\eps t}^\infty e^{-\rho^2} \log \Big(\frac{\rho}{t} + \beta\Big) \, d\rho$$
since for $r > \beta + \eps$ and as $\beta > \alpha$, we know $s(t|r+\alpha|)  + s(t|r-\beta|) \leq 2s(t|r-\beta|)$ as $s(\rho)$ is decreasing in $\rho$. Since $\log(\rho / t + \beta) \leq \log(\rho + \beta)$ whenever $t = \frac{\sqrt{\log g}}{4 \sqrt{3}} > 1$, that is, for large enough genus $g$, then using $\eps = 1/t$, we obtain a bound
$$\int_{\beta + \eps}^\infty |H_t(r) - \tilde \1_{[\alpha,\beta]}(r)| \Big|\frac{\Gamma'}{\Gamma}(1+ ir)\Big| \, dr \lesssim \frac{1}{t}\int_{0}^\infty e^{-\rho^2} \log \Big(\rho+ \beta\Big) \, d\rho \lesssim_\beta \frac{1}{\sqrt{\log g}}$$
as the integral over $[0,\infty]$ is a finite constant depending on $\beta$.

We can repeat this idea for the other parts of the integration. If $r < \alpha - \eps$, we know that $|H_t(r) - \tilde \1_{[\alpha,\beta]}(r)|\leq s(t|r+\alpha|) + s(t|r-\alpha|)$. Depending on now whether $r$ is positive or negative, we use a different bound. In case here $r \geq 0$, as $s$ is decreasing, we will use the bound $s(t|r+\alpha|) + s(t|r-\alpha|) \leq 2s(t|r-\alpha|)$. In case $r < 0$, we use instead $s(t|r+\alpha|) + s(t|r-\alpha|) \leq 2s(t|r+\alpha|)$. Therefore, we can always bound:
\begin{align*}&\int_{(-\infty,\alpha-\eps] \cap C_\eps} |H_t(r) - \tilde \1_{[\alpha,\beta]}(r)| \Big|\frac{\Gamma'}{\Gamma}(1+ ir)\Big| \, dr\\
& \lesssim \int_{-\infty}^{-\alpha-\eps} s(t|r+\alpha|)|\log |r|| \, dr +  \int_{-\alpha + \eps}^{\infty} s(t|r+\alpha|)|\log |r|| \, dr + \int_{-\infty}^{\alpha - \eps} s(t|r-\alpha|)|\log |r|| \, dr \\
& \lesssim_{\alpha} \frac{1}{\sqrt{\log g}},
\end{align*}
where the last inequality is performed exactly like before with the integral over $(\beta+\eps,\infty)$.

Finally, we need to bound the integral over $[\alpha+\eps,\beta-\eps]$. Here $|H_t(r) - \tilde \1_{[\alpha,\beta]}(r)| \leq s(t|r+\alpha|)  + s(t|r-\alpha|) +  s(t|r-\beta|)$. Now, setting $r_0 := (\beta-\alpha)/2$, we know that $|r_0-\alpha| = |r_0-\beta|$ so for $r < r_0$, we will use the bound $|H_t(r) - \tilde \1_{[\alpha,\beta]}(r)| \leq 3s(t|r-\alpha|)$ and for $r \geq r_0$, the bound $|H_t(r) - \tilde \1_{[\alpha,\beta]}(r)| \leq 3s(t|r-\beta|)$, which are possible as $s(\rho)$ is decreasing. These give us
\begin{align*}&\int_{\alpha + \eps}^{\beta-\eps} |H_t(r) - \tilde \1_{[\alpha,\beta]}(r)| \Big|\frac{\Gamma'}{\Gamma}(1+ ir)\Big| \, dr\\
& \lesssim \int_{\alpha + \eps}^{\infty} s(t|r-\alpha|)|\log |r|| \, dr +  \int_{-\infty}^{\beta-\eps} s(t|r-\beta|)|\log |r|| \, dr \\
& \lesssim_{\alpha,\beta} \frac{1}{\sqrt{\log g}}
\end{align*}
as with the other cases.
\end{proof}

For $t = \frac{\sqrt{\log g}}{4 \sqrt{3}}$, the previous lemmas give us
\begin{align*} 
\frac1{|X|} \left( \sum_j H_t(r_j) + \frac1{4\pi} \int_{-\infty}^{\infty} H_t(r) \frac{-\varphi'_X}{\varphi_X}\Big(\frac12 + ir\Big) \, dr \right)
 = \frac{1}{4\pi} \int_{\alpha}^{\beta} r \, \tanh(\pi r) \, dr + O\left(\sqrt{ \frac{b}{\log g}} \right)
\end{align*}
We now need to take special care of the complex values $r_j \in \C$ in the discrete sum on the left-hand side. This is because the test function $H_t$ is not real valued and small for complex values. We use a bound on the number of complex $r_j$, or equivalently on the number of eigenvalues $\leq \frac14$. We remark that this number of eigenvalues is $\leq 2g - 2 + k(g) = O(|X|)$ by a result of Otal and Rosas \cite{OtalRosas}. This gives the following lemma from \cite[Lemma 26]{Monk} with an identical proof, because it only uses the fact that the number of complex $r_j$ is of the order of the volume of $X$. Note here is the only place where the assumption $a \geq \frac{1}{2}$ is needed.

\begin{lemma}\label{lma:12need} Assume now that $\frac{1}{2} \leq a < b$. If $X \in \cA_{g,k(g)}$ with $k(g) = O(g^{\kappa})$ for some $0 \leq \kappa < 1/2$, then
$$\frac1{|X|} \sum_{r_j \notin \R} H_t(r_j) \lesssim \frac1{t}.$$
\end{lemma}

Using this, we end up with the following proposition, similar to \cite[Corollary 27]{Monk}.
\begin{prop}\label{p:tracetest}
There exists $g_0 \in \N$ such that for any $g \geq g_0$, any $\frac{1}{2} \leq a < b$ and any hyperbolic surface $X \in \cA_{g,k(g)}$ with $k(g) = O(g^{\kappa})$ for some $0 \leq \kappa < 1/2$, we have
\begin{align*} 
&\Big|\frac1{|X|} \left( \sum_{r_j \in \R} H_t(r_j) + \frac1{4\pi} \int_{-\infty}^{\infty} H_t(r) \frac{-\varphi'_X}{\varphi_X}\Big(\frac12 + ir\Big) \, dr \right) -  \frac{1}{4\pi} \int_{a}^{b}  \tanh(\pi \sqrt{\lambda - \tfrac{1}{4}}) \, d\lambda \Big|\\
 & = O\left(\sqrt{ \frac{b}{\log g}} \right),
\end{align*}
for $t = \frac{\sqrt{\log g}}{4 \sqrt{3}}$, $\alpha = \sqrt{a - \frac{1}{4}}$ and $\beta = \sqrt{b - \frac{1}{4}}$, where the implied constant only depends on the exponent $\kappa$.
\end{prop}

Proposition \ref{p:tracetest} allows us to control the part of the spectrum in $[\frac{1}{2},\infty)$. However, to prove Theorem \ref{thm:sp2}, we would also need to consider the spectrum in $[0,\frac{1}{2})$ for which we need a different test function that gives an analogue of Proposition \ref{p:tracetest} since $H_t$ does not have good estimates here. This is done similarly as in Monk \cite[Section 3.2]{Monk}, where we can use the analytic and even test functions for $0 \leq a < b \leq 1$ defined by
$$ \tilde H_t(r): = f_t\left(\frac14 + r^2\right)$$
where for any $\lambda \geq 0$
$$ f_t(\lambda) := (\mathbf{1}_{[a,b]} \ast v_t)(\lambda) = \frac{t}{\sqrt{\pi}} \int_a^b \exp(-t^2(\lambda - \mu)^2) \, d\mu.$$
In \cite[Section 3.2]{Monk} Monk used the notation $h_t$ for $\tilde H_t$ but we want to avoid confusion with the notation in our earlier definition $H_t$. The proof follows virtually the same steps and adapts in the same way, using same Selberg's trace formula (Lemma \ref{lma:SelbergTrace}) as we did above, all the relevant estimates associated with this new test function are in \cite[Lemma 11]{Monk} as we did with $H_t$ above using Lemma \ref{lma:decayHt}, which gives us the analogue of \cite[Corollary 18]{Monk}:

\begin{prop}\label{p:tracetestf}
There exists $g_0 \in \N$ such that for any $g \geq g_0$, any $0 \leq a < b \leq 1$ and any hyperbolic surface $X \in \cA_{g,k(g)}$ with $k(g) = O(g^{\kappa})$ for some $0 \leq \kappa < 1/2$, we have
\begin{align*} 
&\Big|\frac1{|X|} \left( \sum_{r_j \in \R} \tilde H_{\tilde t}(r_j) + \frac1{4\pi} \int_{-\infty}^{\infty} \tilde H_{\tilde t}(r) \frac{-\varphi'_X}{\varphi_X}\Big(\frac12 + ir\Big) \, dr \right) -  \frac{1}{4\pi} \int_{a}^{b}  \tanh(\pi \sqrt{\lambda - \tfrac{1}{4}}) \, d\lambda \Big|\\
 & = O\left(\sqrt{ \frac{b}{\log g}} \right),
\end{align*}
for $\tilde t := \frac{\sqrt{\log g}}{62 \sqrt{6}}$, $\alpha = \sqrt{a - \frac{1}{4}}$ and $\beta = \sqrt{b - \frac{1}{4}}$, where the implied constant only depends on the exponent $\kappa$.
\end{prop}

Let us now show how Theorem \ref{thm:sp2} follows from Proposition \ref{p:tracetest}. We will follow a similar strategy as in \cite{Monk}, where we first apply Proposition \ref{p:tracetest} to small intervals $[b_j,b_{j+1}]$ of length $1/t \approx 1/\sqrt{\log g}$, but now due to errors produced by Proposition \ref{lma:comparescattering}, we end up with extra coefficient in front of the constant term that the Gaussian tail of $h_t$ will mitigate in the end.

Let us first give a technical upper bound lemma that follows by combining Proposition \ref{p:tracetest} with Proposition \ref{lma:comparescattering}. It has an error produced by the spectral window that we do not want in Theorem \ref{thm:sp2}, but applying it with short intervals $J$ of length roughly $1/\sqrt{\log g}$ together with the Gaussian tail bounds for the propagators $H_t$ and $\tilde H_t$ to prove Theorem \ref{thm:sp2}.

\begin{lemma}\label{l:spectral estimate}
There exists $g_0 \in \N$ such that for any $g \geq g_0$, any $0 \leq b_1 < b_2$ and any hyperbolic surface $X \in \cA_{g,k(g)}$ with $k(g) = O(g^{\kappa})$ for some $0 \leq \kappa < 1/2$, if $J = [b_1,b_2]$, we have
$$\Big|\frac{N(X,J) + M(X,J)}{|X|} - \frac{1}{4\pi} \int_I  \, \tanh(\pi \sqrt{\lambda - \tfrac{1}{4}}) \, d\lambda\Big| =   O\left(\sqrt{ \frac{b_2}{\log g}} + (\beta_2 - \beta_1) \right)$$
where $b_j = \frac14 + \beta_j^2$, $j = 1,2$, and the implied constant only depends on the exponent $\kappa$. 
\end{lemma}
\begin{proof}
Let us consider first the case $b_1 \geq \frac{1}{2}$ and so define now $H_t$ using the parameters $a = b_1$ and $b = b_2$. Define the signed measure
$$ d\nu(r) := \sum_{r_j \in \R} d\delta_{r_j}(r) + \frac1{2\pi} \frac{-\phi_X'}{\phi_X}\left(\frac12 + ir\right) dr$$
and the positive measure
$$ d\tilde \nu(r) := \sum_{r_j \in \R} d\delta_{r_j}(r) +  \frac1{2\pi} \Big|\frac{-\phi_X'}{\phi_X}\left(\frac12 + ir\right)\Big| dr.$$
Then we have:
$$N(X,J) + M(X,J) = \int 1 \, d\nu \leq \int 1 \, d\tilde \nu.$$
Now Proposition \ref{lma:comparescattering} and positivity of $H_t$ and $\tilde \nu$ implies that for some constant $C > 0$ independent of $X$ we have
\begin{align*}
\frac{N(X,J) + M(X,J)}{|X|} \times \inf_{[\beta_1, \beta_2]} H_t  & \leq \frac{1}{|X|} \int 1 \,d\tilde \nu(r) \times \inf_{[\beta_1, \beta_2]} H_t  \\
&\leq \frac1{|X|} \int H_t(r) \,d\tilde\nu(r) \\
& \leq \frac1{|X|} \Big( \left| \int H_t(r)\, d\nu(r)\right| \\
&\qquad \qquad + C(k(g) \log(|X|) + |X|)\int H_t(r) \, dr \Big).
\end{align*}
Now by Proposition \ref{p:tracetest} applied with $a = b_1$, $b = b_2$ and $t = \frac{\sqrt{\log g}}{4 \sqrt{3}}$, the integral with respect to $\nu$ has the estimate:
$$\frac1{|X|} \int H_t(r) \,d\nu(r) = \frac{1}{4\pi} \int_J \, \tanh(\pi \sqrt{\lambda - \tfrac{1}{4}}) \, d\lambda + O\left(\sqrt{ \frac{b_2}{\log g}} \right).$$
Thus, as $k(g) = O(g^{\kappa})$ for some $0 \leq \kappa < 1/2$, we arrive to:
\begin{align*}
&\frac1{|X|} \left( \left| \int H_t(r) d\nu(r)\right| + (k(g) \log(|X|)+ |X|)\int_0^\infty H_t(r) \, dr \right)\\
& =  \left|\frac1{|X|} \int H_t(r) \,d\nu(r)\right| + \Big(\frac{k(g) \log(|X|)}{|X|} + 1\Big)\int H_t(r) \, dr \\
& \leq \frac{1}{4\pi} \int_J  \, \tanh(\pi \sqrt{\lambda - \tfrac{1}{4}}) \, d\lambda  + O\Big( \sqrt{\frac{b_2}{\log g}} + \int H_t(r) \, dr\Big).
\end{align*}
Finally, using the the estimates on $|H_t(r) - \tilde\1_{[\alpha, \beta]}(r)|$ from Lemma \ref{lma:decayHt} we can deduce that
$\inf_{[\beta_1, \beta_2]} H_t = O(1)$ and $\int H_t(r) \, dr = O(\beta_2 - \beta_1)$, which completes the proof when $b_1 \geq \frac{1}{2}$. Now when $b_2 \leq 1$, we can just repeat the above argument by using $\tilde H_t$ instead and apply Proposition \ref{p:tracetestf} instead with $t =  \frac{\sqrt{\log g}}{4 \sqrt{3}}$ replaced by $\tilde t = \frac{\sqrt{\log g}}{62 \sqrt{6}}$, which gives the claim.
\end{proof}

Returning now to the statement of Theorem \ref{thm:sp2}, the upper and lower bounds on $R(X,I)$ are then proved in a similar way as in \cite[Section 3.5]{Monk}, but we need to be careful with the continuous part of the spectral density and use Proposition \ref{lma:comparescattering}. We give the proof of the lower bound, which is the part we really need in this paper.

\begin{proof}[Proof of the lower bound for $R(X,I)$ in Theorem \ref{thm:sp2}]
Consider first the case $\frac{1}{2} \leq a < b$. Then $\alpha,\beta$ determined by $a = \frac14 + \alpha^2$ and $b = \frac14 + \beta^2$ satisfy $\frac{1}{2}\leq \alpha < \beta$. To prove the lower bound in Theorem \ref{thm:sp2}, we need to find $C > 0$ independent of $X$ and $a,b$ such that
$$\frac{N(X,I) + M(X,I)}{|X|} \geq \frac{1}{4\pi} \int_I  \, \tanh(\pi \sqrt{\lambda - \tfrac{1}{4}}) \, d\lambda - C\sqrt{\frac{b+1}{\log g}}.$$
Let us split
\begin{align*}
\frac{N(X,I) + M(X,I)}{|X|} = &\,\,\frac{1}{|X|} \int H_t \, d\nu - \frac{1}{|X|} \int_\beta^\infty H_t \, d\nu - \frac{1}{|X|} \int_{-\infty}^\alpha H_t \, d\nu \\
&- \frac{1}{|X|} \int (H_t-\1_{[\alpha,\beta]}) \, d\nu
\end{align*}
Let us now bound each of these term individually. For the first term, $\frac{1}{|X|} \int H_t \, d\nu$, we apply Proposition \ref{p:tracetest}, which gives us
$$\frac1{|X|} \int H_t(r) \, d\nu(r) \geq \frac{1}{4\pi} \int_I  \, \tanh(\pi \sqrt{\lambda - \tfrac{1}{4}}) \, d\lambda - O\left( \sqrt{ \frac{b}{\log g}}\right)$$
for an implied constant independent of $X$, $a$ and $b$. Therefore, we will be done, if we can bound all the other terms in absolute value above by a constant multiple of $\sqrt{\frac{b+1}{\log g}}$ using the Gaussian tail of $h_t$ defining $H_t$.

Let us first look at in detail bounding the term $\frac{1}{|X|} \int_\beta^\infty H_t \, d\nu$. For this purpose, we apply Lemma \ref{l:spectral estimate} on intervals of length $\frac1{\sqrt{\log g}}$. More precisely following \cite{Monk}, for the case $r \geq \beta$ we use a subdivision $b_j = b + \frac{j}{t}$, with $t = \frac{\sqrt{\log g}}{4 \sqrt{3}}$ and writing $\beta_j = \sqrt{b_j - \frac14}$. If we now use the trivial bound $|\int_{\beta_j}^{\beta_{j+1}} H_t \, d\nu| \leq \int_{\beta_j}^{\beta_{j+1}} H_t \, d\tilde \nu$ since $H_t \geq 0$, together with Lemma \ref{l:spectral estimate} applied to intervals $J = J_j = [b_j,b_{j+1}]$, we obtain an upper bound:

\begin{align*}
\frac1{|X|} \Big|\int_\beta^{\infty} H_t(r) \, d \nu(r)\Big| &\leq \sum_{j=0}^\infty \frac1{|X|}\Big| \int_{\beta_j}^{\beta_{j+1}} H_t(r) \, d \nu(r)\Big|\\
&\leq  \sum_{j=0}^\infty \frac1{|X|} \int_{\beta_j}^{\beta_{j+1}} H_t(r) \, d\tilde \nu(r)\\
&\leq \sum_{j=0}^\infty \frac{N(X,J_j) + M(X,J_j)}{|X|} \times \sup_{[\beta_j,\beta_{j+1}]} H_t\\
&= O\left(  \sum_{j=0}^\infty  \left( b_{j+1} - b_j + \sqrt{\frac{b_{j+1} + 1}{\log g}} + \beta_{j+1} - \beta_j \right) \times \sup_{[\beta_j,\beta_{j+1}]} H_t \right)\\
& = O\left(  \frac1{\sqrt{\log g}} + \frac1{\sqrt{\log g}}  \sum_{j=1}^\infty \sqrt{j} \sup_{[\beta_j,\beta_{j+1}]} H_t \right),
\end{align*}
where in last line we used the mean value theorem for square root to give $\beta_{j+1} - \beta_j \leq \frac{1}{2\beta_j}(b_{j+1}-b_j) \leq b_{j+1} - b_j$ as $\beta_j \geq \beta > \alpha \geq \frac{1}{2}$ by the assumption $a \geq \frac{1}{2}$ and also the uniform bound $\|H_t\|_\infty \leq 2$ used in the interval $[\beta_0,\beta_1]$. Now by Lemma \ref{lma:decayHt}, if $r \in [\beta_j,\beta_{j+1}]$, and $j \geq 1$, we have:
$$ |H_t(r)| \leq \frac{e^{-\sqrt{j}}}{2\sqrt{\pi} \sqrt{j}}.$$
Therefore we obtain
$$\frac1{|X|} \int_\beta^{\infty} H_t \,d \nu = O\left(\frac1{\sqrt{\log g}}\right).$$
Since the support of $\nu$ is contained in $[0,\infty)$, the other integrals $\frac1{|X|} \int_{-\infty}^{\alpha} H_t \,d \nu = \frac1{|X|} \int_{0}^{\alpha} H_t \,d \nu$ and $\frac{1}{|X|} \int (H_t-\1_{[\alpha,\beta]}) \, d\nu = \frac{1}{|X|} \int_0^\infty (H_t-\1_{[\alpha,\beta]}) \, d\nu$. They can be then dealt with an identical argument by splitting the integration domain inside $[0,\infty)$ into intervals of length $1/t$ and applying Lemma \ref{l:spectral estimate} together with the estimates from Lemma \ref{lma:decayHt}. Thus we have finished the proof in the case of $a \geq \frac{1}{2}$. Now if $0 \leq a < b \leq 1$, we can use the similar argument as above, but now with the test function $\tilde H_t$ and Proposition \ref{p:tracetestf} instead, which works for the subintervals of $[0,1]$.
\end{proof}

\section{Rate of quantum ergodicity on random surfaces}\label{sec:randomsurfaces}

Let us now discuss in detail how we can prove quantum ergodicity for eigenfunctions of the Laplacian on random surfaces in the large genus limit (Theorem \ref{thm:largegenus}). The proof of this follows from Theorem \ref{thm:main} together with Theorem \ref{thm:MMH} mentioned in the introduction and Theorem \ref{thm:sp2}. 

\begin{proof}[Proof of Theorem \ref{thm:largegenus}]Fix a compact interval $I \subset (1/4,+\infty)$ and large enough $g \geq 2$. Assume $k(g)$ satisfies $k(g) = O(g^{\kappa})$ for some $0 \leq \kappa < 1/2$. Fix now a small enough $\eps > 0$ such that $\frac{1}{4} - \Big(\frac{2\kappa+1}{4}\Big)^2 - \eps > 0$. Recall we defined that any $X \in \cB_{\eps,\kappa,g,k(g)} \subset \cM_{g,k(g)}$ satisfies the three conditions:
$$\frac{|(X)_{\leq \frac16 \log g}|}{|X|} \leq g^{-\frac13}, \quad \mathrm{sys}(X) \geq g^{-\frac1{24}}(\log g)^{\frac12} \quad \text{and} \quad \lambda_1(X) \geq \frac{1}{4} - \Big(\frac{2\kappa+1}{4}\Big)^2 - \eps,$$
Moreover, by Theorem \ref{thm:MMH}, there exists $\beta > 0$ such that
$$\P_{g,k(g)}(\cB_{\eps,\kappa,g,k(g)}) \geq 1-O_{\eps,\kappa}(g^{-\beta}), \quad \text{as } g \to \infty.$$
Thus Theorem \ref{thm:largegenus} follows, if we can establish that for all surfaces $X \in \cB_{\eps,\kappa,g,k(g)}$ and $a \in L^\infty_{Y}(X)$ for $Y := \log g$, the quantum mean absolute deviation
$$\Var_{X,I}(a) \lesssim_{I,\kappa} \frac{1}{\sqrt{\log g}} \|a\|_\infty$$
 with an implied constant \textit{independent} of $X$. Recall here $L^\infty_{Y}(X)$ is the set of $a \in L^\infty(X)$ such that the support of $a$ satisfies $\spt a \subset X(Y)$, where $ X(Y) = X \setminus \bigcup_{\mathfrak{b}} X_{\mathfrak{b}} (Y),$ and $X_{\mathfrak{b}} (Y)$ is the cuspidal zone associated with $\mathfrak{b}$.

Thus let us fix $X \in \cB_{\eps,\kappa,g,k(g)}$, $a \in L^\infty_Y(X)$ for $Y := \log g > 0$ as $g \geq 2$, and a compact interval $I \subset (\frac{1}{4},\infty)$, and bound $\Var_{X,I}(a)$ from above.

First of all, by Theorem \ref{thm:main}, there exists $R_I > 0$ such that for all $R > R_I$ such that
\begin{align*}
 \widetilde \Var_{X,I}(a) 
& \lesssim_I   \Big( \max\{N(X,I), k(g)\}^{1/2}\left( \frac{ |X|}{ \,\rho(\lambda_1(X)) R} +\frac{ e^{2R}}{\min\{1,{\inj_{X(Ye^{R/2})}}^2 \}} |(X)_{\leq R} | \right)^{1/2}  \\
&\quad +  2k(g) \log Y + k(g)^2 e^{-4\pi Y} +  \frac{k(g)}{|X|} \left(M(X,I) + k(g) \log(|X|) \Big) \right)\|a\|_\infty,
\end{align*}
where $\widetilde \Var_{X,I}(a)  = (N(X,I) + M(X,I))\Var_{X,I}(a)$ and $\rho(\lambda_1(X))$ is a function of the spectral gap of $X$. We will apply this bound with the choice $R := \frac{1}{16} \log(g)$, where we assume $g \geq 2$ is large enough such that $R > R_I$. 

Divide now the upper bound for $\widetilde \Var_{X,I}(a)$ by $N(X,I) + M(X,I)$, which leads to the following estimate:
\begin{align*}
\Var_{X,I}(a) 
& \lesssim_I   \Big[ \Big(\underbrace{\frac{\max\{N(X,I), k(g)\}}{N(X,I) + M(X,I)}}_{\text{Term (a)}}\Big)^{1/2} \left( \underbrace{\frac{\frac{ |X|}{ \,\rho(\lambda_1(X)) R} +\frac{ e^{2R}}{\min\{1,{\inj_{X(Ye^{R/2})}}^2 \}} |(X)_{\leq R} |}{N(X,I) + M(X,I)}}_{\text{Term (b)}} \right)^{1/2}  \\
&\quad +  \underbrace{\frac{2k(g) \log Y + k(g)^2 e^{-4\pi Y} +  \frac{k(g)}{|X|} \Big(M(X,I) + k(g) \log(|X|) \Big)}{N(X,I) + M(X,I) }}_{\text{Term (c)}}\Big] \|a\|_\infty,
\end{align*}
where we have indicated three terms (a), (b) and (c) that we will estimate now in the following.

\textit{Term (a).} First of all, Theorem \ref{thm:sp2} on quantitative spectral convergence implies that 
\begin{align}\label{eq:NMbound} N(X,I) + M(X,I) \geq |X| \left( c_I + R(X,I) \right),\end{align}
where
$$c_I := \frac{1}{4\pi}\int_{1/4}^\infty \1_I(\lambda) \tanh(\pi \sqrt{\lambda-1/4}) \, d\lambda \quad \text{and}\quad R(X,I) \gtrsim_{I,\kappa} - \sqrt{\frac1{\log g}}.$$
Thus, as $k(g) = O(g^{\kappa})$ for $\kappa < 1/2$, $|X| = O(g)$, we can bound the term (a) as follows:
$$\frac{\max\{N(X,I), k(g)\}}{N(X,I) + M(X,I)} =  \max\Big\{\frac{N(X,I)}{N(X,I) + M(X,I)}, \frac{k(g)}{N(X,I) + M(X,I)}\Big\}   \lesssim_{I,\kappa} 1.$$

\textit{Term (b).} By definition in the $\cB_{\eps,\kappa,g,k(g)}$ we have the following $X$ independent uniform spectral gap bound from below:
$$\lambda_1(X) \geq \frac{1}{4} - \Big(\frac{2\kappa+1}{4}\Big)^2 - \eps$$
for all $X \in \cB_{\eps,\kappa,g,k(g)}$. On the other hand, like we discussed in Remark \ref{rmk:unifsp}, this implies that $\rho(\lambda_1(X)) \geq c_0$ for all $X \in \cB_{\eps,\kappa,g,k(g)}$, where $c_0 > 0$ is independent of the surface $X$.

Moreover, as $Y = \log g$, $R = \frac{1}{16}\log (g)$, and $\mathrm{sys}(X) \geq g^{-\frac1{24}}(\log g)^{\frac12}$, we have for large enough $g$:
$$\inj_{X(Ye^{R/2})} = \frac{1}{2}\min\{\sys(X),e^{-R/2} Y^{-1}\} \gtrsim g^{-\frac1{24}}(\log g)^{\frac12}$$
as $e^{-R/2} = g^{-\frac{1}{32}}$. Furthermore, since $R = \frac{1}{16} \log(g)$, we have 
$$(X)_{\leq R} \subset (X)_{\leq \frac16 \log g}$$
giving
$$\frac{|(X)_{\leq R}|}{|X|} \leq \frac{|(X)_{\leq \frac16 \log g}|}{|X|} \leq g^{-\frac13}.$$ 
Therefore, by \eqref{eq:NMbound} and $e^{2R} = e^{2 \cdot \frac{1}{16} \log(g)} = g^{1/8}$ we obtain the following bound for the term (b):
\begin{align*} \frac{\frac{ |X|}{ \,\rho(\lambda_1(X)) R} +\frac{ e^{2R}}{\min\{1,{\inj_{X(Ye^{R/2})}}^2 \}} |(X)_{\leq R} |}{N(X,I) + M(X,I)} 
& \lesssim \frac{1}{ R}+ \frac{e^{2R}}{\min\{1,\inj_{X(Ye^{R/2})}^2\}} \frac{|(X)_{\leq R}|}{|X|}  \\
& \lesssim \frac{1}{\log g} + \frac{g^{1/8}}{\min\{1,\inj_{X(Ye^{R/2})}^2\}} \frac{|(X)_{\leq \frac16 \log g}|}{|X|}  \\
& \lesssim  \frac{1}{ \log g} + \frac{g^{1/8}}{g^{-\frac{1}{12}} \log g} g^{-1/3}\\
& \lesssim \frac{1}{ \log g}.
\end{align*}

\textit{Term (c).} Finally, using $Y = \log g$, $k(g) = O(g^{\kappa})$ for $\kappa < 1/2$, $|X| = O(g)$ and the bound \eqref{eq:NMbound}, we can estimate:
$$\frac{2k(g) \log Y + k(g)^2 e^{-4\pi Y} +  \frac{k(g)}{|X|} \left(M(X,I) + k \log|X| \right)}{N(X,I) + M(X,I)} \lesssim_{I,\kappa} \frac{\log\log g}{g^{1-\kappa}}.$$

Combining (a), (b) and (c) gives us the claim
$$\Var_{X,I}(a) \lesssim_{I,\kappa} \Big(1 \cdot \Big( \frac{1}{\log g}\Big)^{1/2} + \frac{\log\log g}{g^{1-\kappa}}\Big)  \|a\|_\infty \lesssim \frac{1}{\sqrt{\log g}} \|a\|_\infty$$
with a constant independent of $X$. Thus the proof of Theorem \ref{thm:largegenus} is complete.
\end{proof}

\section*{Acknowledgements}

We thank Michael Magee, Laura Monk, Paul Nelson, Jean Raimbault, and Abhishek Saha for useful discussion related to the topic. We also thank the anonymous referee for several useful comments on earlier versions of this manuscript.

\appendix
\section{Probability of a small systole}\label{a:systoleapp}

We extend here the result of Mirzakhani \cite[Theorem 4.2]{Mi} on the Weil-Petersson probability of having a small systole, adding a quantitative dependence on the the number of cusps $k$ and giving a proof for non-compact hyperbolic surfaces. This is used for Theorem \ref{thm:sp1}. 

The proof is essentially the same as Mirzakhani's argument, but we rely on Lemma A.4 from Hide \cite{Hide21}. Before, we state and prove the lemma, let us recall some notation from \cite{Mi}. If we have $L = (L_1,\dots,L_k)$ with $L_i \geq 0$, then we define the moduli space $\cM_{g,k}(L)$ as $\cM_{g,k}$ such that the boundary elements are associated with lengths $L_i$, $i = 1,2,\dots,k$. Thus the moduli space corresponding to finite area hyperbolic surfaces with $k$ cusps is given by $\cM_{g,k} = \cM_{g,k}(0,\dots,0)$. In each space $\cM_{g,k}(L)$ one can also consider the Weil-Petersson volume $\mathrm{Vol}_{\mathrm{WP}}$ and we have the following relation of the Weil-Petersson volumes $V_{g,k}(L_1,\dots,L_k)$ of $\cM_{g,k}(L_1,\dots,L_k)$ and $V_{g,k}$ of $\cM_{g,k}$:
$$V_{g,k}(L_1,\dots,L_k) \leq e^{L_1 + \dots + L_k} V_{g,k},$$
see \cite[(3.7)]{Mi}. 

We have the following result.

\begin{lemma}\label{lma:noncompactsystole}
Suppose $k(g) = o(\sqrt{g})$. Suppose $0 < \eps < 1$. Then for any $g \geq 2$:
$$\P_{g,k(g)} (X \in \cM_{g,k(g)} : \sys(X) \leq \eps) \lesssim \eps^2,$$
where the implied constant is independent of $\eps$ and $g$.
\end{lemma}

\begin{proof} 
Fix $0 < \eps < 1$ and $g \geq 2$ and define the event
$$\cM_{g,k(g)}^\eps := \{X \in \cM_{g,k(g)} : \sys(X) \leq \eps\}.$$ 
We just need to verify that
$$\frac{\Vol_{\mathrm{WP}}(\cM_{g,k(g)}^{\eps_g})}{V_{g,k(g)}} \lesssim \eps^2$$
with an implied constant independent of $\eps > 0$. Define the counting function
$$N_{0}(X,\eps) := \sharp \{\gamma \subset X : \ell_\gamma(X) \leq \eps, \gamma \text{ is non-separating}\},$$
that is, $N_0(X,\eps)$ is the number of simple closed geodesics $\gamma$ of length $\leq \eps$ on $X - \gamma$ is a surface of genus $g-1$ and $k(g)$ cusps and $2$ boundary curves. Furthermore, for $i \geq 1$ and $j \geq 0$, we define $N_{i,j}(X,\eps)$ as the number of simple closed geodesics $\gamma \subset X$ of length $\leq \eps$ which divide $X$ into a surface of genus $i$ and $j$ cusps and $1$ boundary curve and a surface of genus $g-i$ with $k(g)-j$ cusps and $1$ boundary curve. Now using
$$N(X,\eps) :=  N_0(X,\eps) + \sum_{i = 1}^{ \lceil g/2\rceil } \sum_{j = 0}^{ k(g)} N_{i,j}(X,\eps).$$
as in the proof of Theorem 4.2 of \cite{Mi} using Mirzakhani's integral formula (Theorem 2.2 \cite{Mi}) we can estimate:
\begin{align*} &\Vol_{\mathrm{WP}}(\cM_{g,k(g)}^{\eps})\leq \int_{\cM_{g,k(g)}} N(X,\eps) \, dX\\
& = \frac{1}{2} \int_0^{\eps} t \Vol_{\mathrm{WP}}(\cM_{g-1,k(g)+2}(0,0,\dots,0,t,t)) \, dt \\
& \qquad + \frac{1}{2}\sum_{i = 1}^{ \lceil g/2\rceil } \sum_{j = 0}^{ k(g) } {k(g) \choose j}   \int_0^{\eps} t \Vol_{\mathrm{WP}}(\cM(S_{i,j+1} \times S_{g-i,k-j+1}(0,0,\dots,0,t,t)) \, dt,
\end{align*}
where the binomial coefficient occurs as we need to select the $j$ cusps from all the $k(g)$ possibilities since the mapping class group fixes the cusps pointwise, which means that they will be in different mapping class group orbits. Now, by $V_{g,k(g)}(L_1,\dots,L_k) \leq e^{L_1 + \dots + L_k} V_{g,k(g)}$ as in the proof of \cite[Theorem 4.2]{Mi} we have:
$$\Vol_{\mathrm{WP}}(\cM_{g-1,k(g)+2}(0,0,\dots,0,t,t)) \leq e^{2t} V_{g-1,k(g)+2}.$$
and
$$\Vol_{\mathrm{WP}}(\cM(S_{i,j+1} \times S_{g-i,k(g)-j+1}(0,0,\dots,0,t,t)) \leq e^{2t} V_{i,j+1} V_{g-i,k(g)-j+1}.$$
By Lemma 5.1(iii) of Mirzakhani and Zograf \cite{MZ15}, there is a universal constant $C_2 > 0$ such that as long as $k(g) = o(\sqrt{g})$, we have as $g \to \infty$:
$$\frac{V_{g-1,k(g)+2} }{V_{g,k(g)}} \leq 1-C_2 \frac{k(g)-2}{2g - 4 + k(g)} = O(1).$$
Furthermore, by Lemma A.4 of Hide \cite{Hide21} there exists a universal constant $C_1 > 0$ such that for any sequence $(k(g))_{g \geq 2}$ with $k(g) = o(\sqrt{g})$, we have
$$\sum_{0 \leq i \leq g, 0 \leq j \leq k(g) \atop 2 \leq 2i + j \leq 2g + k(g) - 2}{ k(g)\choose j} \frac{ V_{i,j+1}V_{g-i,k(g) - j +1}}{V_{g,k(g)}} \leq C \frac{1+k(g)^2}{g} = o(1)$$
since whenever $1 \leq i \leq \lceil g/2\rceil$ and $0 \leq j \leq k(g)$, then $2 \leq 2i + j \leq 2g + k(g) - 2$ since $g \geq 2$. Thus
$$\frac{\Vol_{\mathrm{WP}}(\cM_{g,k(g)}^{\eps})}{V_{g,k(g)}} \lesssim \eps^2 e^{2\eps} \lesssim \eps^2$$
as $\eps < 1$, where the implicit constant is independent of $\eps$, $k(g)$ and $g$.
\end{proof}

\bibliographystyle{plain}
\bibliography{quantumgenus_rev}

\end{document}